\documentclass[11 pt,a4paper,dvips]{article}
\usepackage{amsmath,amssymb,amsfonts,amsthm,epsfig,graphics,graphicx}

\usepackage[usenames,dvipsnames]{color}
 \def\NN{{\mathbb N}}  
 \def\RR{{\mathbb R}} \def\SS{{\mathbb S}} 
   
 \def\ZZ{{\mathbb Z}}

\def\Wi{\widetilde}

  \def\cH{{\cal H}}  
   \def\cO{{\cal O}}

\def\cF{{\cal F}}  \def\cL{{\cal L}}  \def\cX{{\cal X}}

\newtheorem{theorem}{{Theorem}}[section]
\newtheorem{proposition}[theorem]{{Proposition}}

\newtheorem{lemma}[theorem]{{Lemma}}

\newtheorem{corollary}[theorem]{{Corollary}}

\newtheorem{claim}[theorem]{{Claim}}
\newtheorem{question}[theorem]{{Question}}

\theoremstyle{definition}
\newtheorem{definition}[theorem]{{Definition}}

\theoremstyle{remark}
\newtheorem{remark}[theorem]{{Remark}}

\title{Classifying expanding attractors on figure eight knot complement space and non-transitive Anosov flows on Franks-Williams manifold}
\author{Jiagang Yang and Bin Yu}
\date{\today}

\begin{document}

\maketitle

\begin{abstract}
The path closure of  figure eight knot complement space, $N_0$, supports a natural DA (derived from Anosov) expanding attractor.
Using this attractor, Franks-Williams constructed the first example of non-transitive Anosov flow on the manifold $M_0$ obtained by gluing two copies of  $N_0$ through identity map along their boundaries, named by Franks-Williams manifold. In this paper, our main goal is  to classify  expanding attractors on $N_0$ and non-transitive Anosov flows on $M_0$. We prove that, up to orbit-equivalence,  the DA expanding attractor is the unique expanding attractor supported by $N_0$, and  the non-transitive
Anosov flow constructed by Franks and Williams is the unique non-transitive Anosov flow  admitted by $M_0$.
Moreover,  more general cases are also discussed. In particular, we  completely classify non-transitive  Anosov flows on a family of infinitely many toroidal $3$-manifolds with two hyperbolic pieces, obtained by gluing two copies of $N_0$ through any gluing homeomorphism.
\end{abstract}

\section{Introduction}\label{s.int}
Anosov systems (flow and diffeomorphism) generalize the classical example of geodesic flows on closed Riemannian manifolds with negative curvature.
They were originally called U-system by D.V.Anosov in his celebrated paper \cite{An}, where he  proved that every Anosov system is both structure stable and ergodic. Such systems play a fundamental role in Smale's picture about structure stable system \cite{Sm}.

Since then, Anosov system has become an important mathematical object studied by many people from several different viewpoints.
In particular, many conceptions developed from Anosov system take  bridges between dynamical system, geometry and topology.
One  natural direction is to qualitatively understand Anosov system. Nevertheless, many works have been done in this direction. For instance, for $3$-dimensional diffeomorphisms, a complete topological classification is obtained by Franks \cite{Fra}. His theorem implies that every $3$-dimensional Anosov diffeomporphism is topologically conjugate to some linear Anosov automorphism on $T^3$, up to finite cover. In particular, every $3$-dimensional Anosov diffeomporphism is
transitive. It is still a long-standard open question by Smale that whether   every  Anosov diffeomporphism is transitive.

But for $3$-dimensional Anosov flows, a complete classification seems to be far from reached at present.
There are many non-algebraic Anosov flows constructed, see for instance, \cite{HT}, \cite{Go}, \cite{BL} and \cite{BBY}.
In particular, Franks and Williams \cite{FW} built the first non-transitive Anosov flow on a closed $3$-manifold.
This manifold is the toroidal $3$-manifold by gluing two copies of figure eight knot complement space along their boundaries by identity map. We refer to it as \emph{Franks-Williams manifold} in this paper.
This example is well-known because it gives a negative  answer for
 Smale's orginal question  in the case of Anosov flows: whether every  Anosov system is transitive.

The classification of $3$-dimensional Anosov flows has many deep relations with the topolgy of the background manifolds.
For certain classes of $3$-manifolds, there are already complete classifications:
\begin{itemize}
  \item Plante \cite{Pl} and Ghys \cite{Gh1}  classified Anosov flows on torus bundle over circle and circle bundle over surfaces respectively.
  \item Ghys \cite{Gh2} classified Anosov flows on $3$-manifolds with smooth stable/unstable bundles.
  \item Barbot \cite{Bar} classified Anosov flows on a class of graph manifolds, which he called generalized Bonatti-Langevin manifolds.
\end{itemize}
By lifting Anosov flow to the universal cover of the background manifold,  Barbot and Fenley developed
some powerful tools to understand Anosov flows on $3$-manifolds and used them to reveal many new and deep dynamical and topological behaviors about Anosov flows, see for instance, \cite{Fen1}, \cite{Bar} and \cite{BF1}.  These works have immensely improved our understanding about Anosov flows on $3$-manifolds.
For a more comprehensive discussion on this topic, we suggest Barthelme's nice survey \cite{Bart}.

\subsection{Main results}\label{ss.main}
All of the classification works above focused on transitive Anosov flows.
In this paper, as one of the main results of this paper, we provide a complete classification of non-transitive flows on certain $3$-manifolds,
including Franks-Williams manifold. To our knowledge, this is the first time such classification result
is obtained for non-transitive Anosov flows.

To build a non-transitive Anosov flow on a $3$-manifold, expanding attractor with a standard neighborhood
is necessary. Expanding attractors are  uniformly hyperbolic attractors of flows whose topological dimension is equal to
the dimension of unstable manifold.
We  say that a compact $3$-manifold $N$ supports an expanding attractor $\Lambda$ if the flow is transverse to $\partial N$ with maximal invariant set coincide with $\Lambda$.
More definitions and properties about expanding attractors can be found in Section \ref{ss.epa}.

First we introduce a class of expanding attractors, one of which was used to build the first non-transitive
Anosov flow  by Franks and Williams \cite{FW}.
Starting with a suspension Anosov flow on the sol-manifold
$W_A=T^2 \times [0,1]/ (x,1)\sim (A(x),0)$ induced by  the vector field $(0, \frac{\partial}{\partial t})$,
 we perform a DA-surgery  on a small tubular neighborhood of the periodic orbit associated to the origin in $T^2$. One obtains a new Axiom A flow $Y_t^A$  on $W_A$  with one expanding attractor $\Lambda_A$ and one isolated repeller $\gamma_A$.
  By cutting a suitable small tubular neighborhood of
$\gamma_A$,  a filtrating neighborhood $N_A$ of $\Lambda_A$ is obtained, which is a compact
$3$-manifold with a once-punctured torus fibration structure.
When $A=\left(
                                                                                   \begin{array}{cc}
                                                                                     2 & 1 \\
                                                                                     1 & 1 \\
                                                                                   \end{array}
                                                                                 \right)$,
we use $W_0$, $Y_t^0$, $\Lambda_0$ and $N_0$ instead of $W_A$, $Y_t^A$, $\Lambda_A$ and $N_A$
respectively. Note that $N_0$ is homeomorphic to the figure eight knot complement space.

Now we briefly  introduce Franks-Williams' non-transitive Anosov flow. For more details, see \cite{FW}.
Choose two copies of $(N_0, Y_t^0, \Lambda_0)$, say $(N_1, Y_t^1, \Lambda_1)$ and $(N_2, Y_t^2, \Lambda_2)$, and then choose  a gluing homeomorphism $\Psi_0: \partial N_2 \to \partial N_1$
which is isotopic to identity and satisfies some foliations transversality property.
 The details about this property can be found in Section \ref{s.nonAno}.
 Franks and Williams proved that the glued flow $Z_t^0$ is a non-transitive Anosov flow on the glued manifold $M_0$, which  is called \emph{Franks-Williams manifold}.

Our first main result provides a complete classification of non-transitive Anosov flows on Franks-Williams manifold $M_0$.

\begin{theorem}\label{t.claAno}
Every non-transitive Anosov flow on Franks-Williams manifold $M_0$ is topologically equivalent to $Z_t^0$,
which is the standard model built by Franks-Williams.
\end{theorem}

To prove Theorem \ref{t.claAno}, it is crucial
to classify expanding attractors on $N_0$.

There are several reasons to highlight the study of expanding attractors.
First, every non-transitive Anosov flow admits at least one expanding attractor as its basic set, by Smale's spectral decomposition. The local dynamical property of expanding attractor reflects some global information of non-transitive
Anosov flow.  Moreover, similar to the case of Anosov flow, expanding attractor is interesting in itself due to   its  geometric and topological
properties among its filtrating neighborhood. For instance:
 \begin{enumerate}
   \item the stable foliations of the attractors  is a special class of taut foliations and each attractor itself is an essential lamination in the sense of Gabai-Ortel \cite{GO}, and moreover,
the stable foliation is transverse to the unstable lamination, which provides rich
information about the topology of the background manifold;
   \item Ghrist  showed in his remarkable paper \cite{Ghr} that  the periodic orbits in $(N_0, Y_0^t, \Lambda_0)$
   contain any knot or link type. \footnote{Birman and Williams \cite{BW} firstly studied
   the knot and link types in the union of the periodic orbits carried by the attractor $\Lambda_0$. Their breakthrough  is to construct a  branched surface with semi-flow, called by template nowadays,  so that the union of the periodic orbits carried by the template is in one-to-one correspondence to the union of the periodic orbits carried by $\Lambda_0$. Moreover, they also showed that the correspondence preserves the corresponding knot or link types. Ghrist \cite{Ghr} shows that  every knot or link type is realized in the union of the periodic orbits carried by the template.}
 \end{enumerate}

There are a few works devoted to  expanding attractors of flows on $3$-manifolds.
For instance, Christy \cite{Ch1}, \cite{Ch2} and Section $9$ of \cite{BBY}:
\begin{enumerate}
  \item  \cite{Ch1}  focuses on the existence of Birkhoff section for coherent expanding attractor;\footnote{One can find the definition of coherent expanding attractor in Section \ref{sss.epaflow}.}
  \item  \cite{Ch2} is devoted
to build a type of branched surfaces with semi-flows that carry expanding attractors;
  \item in Section $9$ of \cite{BBY}, the authors constructed many new examples, include many incoherent expanding attractors.
\end{enumerate}

In this paper, we will show that the figure eight knot complement space $N_0$ only carries
a unique expanding attractor, which is the DA attractor $\Lambda_0$.
This result will not only be a progress
in the study of expanding attractors, but also is the first step towards the proof of Theorem
\ref{t.claAno}. Let us represent it more precise here.

\begin{theorem}\label{t.claexp}
Up to topological equivalence, there is a unique
flow  with expanding attractor supported by
the figure eight knot space $N_0$, which
is the DA flow $Y_t^0$ restricted on $N_0$.
\end{theorem}

Now we sketch the proofs of Theorem  \ref{t.claAno} and Theorem \ref{t.claexp},
starting with the proof of the latter.
Every expanding attractor supported by $N_0$ is a $2$-dimensional lamination.
By filling a solid torus to $N_0$ in a standard way, one obtains the sol-manifold $W_0$.
In the mean time, the expanding attractor is a $2$-dimensional lamination on $W_0$.
The first step to prove Theorem \ref{t.claexp} is to show that
there exists a finite cover of $W_0$ so that the lift of the attractor can be extended
to a $C^0$ taut foliation on the covering manifold (Theorem \ref{t.folext}).
We will  prove Theorem \ref{t.folext} in two sides by discussing a type of boundary slopes implied by the attractor:
 \begin{enumerate}
   \item in some cases we will extend the attractor laminations to foliations up to finite cover by using a standard construction known to topologists as the
filling monkey saddle foliation;
   \item in the other cases we will find some obstructions by using Poincare-Hopf Theorem or some special
   topological information about $N_0$.
 \end{enumerate}
 All of these can be found in Section \ref{s.filfol}.

 The  flow in question on $N_0$ can be naturally extended to an Axiom A flow on $W_0$.
 The second step to prove Theorem \ref{t.claexp} is to show that
 there exists a global section for this Axiom A flow (Proposition \ref{p.sect}).
 The proof of Proposition \ref{p.sect} is carried out in the following three steps.
 \begin{enumerate}
   \item By using Theorem \ref{t.folext} and some topological information about the  attractors,
   we  show that every periodic orbit in the attractor essentially intersects with every torus fiber
   of the unique torus fibration structure on $W_0$ (Proposition \ref{l.openper}).
   \item By using the shadowing lemma and Proposition \ref{l.openper}, we can extend the intersection property:  every flowline of the Axiom flow essentially intersects with every torus fiber in $W_0$
       (Proposition \ref{p.liminf}).
   \item By a beautiful theorem of Fuller \cite{Fu} (Theorem \ref{t.Fuller}), we find the desired global section (Proposition \ref{p.sect}).
 \end{enumerate}

 Depending on Proposition \ref{p.sect}, the problem  is transformed
 to the classification of a special class of Axiom A diffeomorphisms on a torus.
 We will use a standard DA diffeomorphism as a model,
 and then show that each of the  diffeomorphism in question is topologically conjugate
 to this model by building conjugacy map (Proposition \ref{p.DAdiff}).

Now we outline the proof of Theorem \ref{t.claAno}, which occupies Section \ref{s.nonAno}.
By some standard techniques (more details can be found in the proof of Lemma \ref{l.decomp}), we get that the non-wandering set of this Anosov flow is the union of an expanding attractor and an expanding repeller so that each of them can be supported by $N_0$. Therefore,
 the proof of Theorem \ref{t.claAno}  can be naturally divided into two steps: local classification of expanding attractor
and global classification of non-transitive Anosov flow.
Since the first step is already done in Theorem \ref{t.claexp}, we are left to consider the
following question regarding the global topology of the underlying manifold:
how gluing maps between boundaries of two filtrating neighborhoods affect topological equivalent classes of the flow.
For this purpose, by using a non-transitive Anosov flow constructed by Franks-Williams  as a model,
we show that every non-transitive Anosov flow is topologically equivalent to this model
by building an orbit-preserving map. To build such a map, there are two points:
\begin{enumerate}
  \item lifting everything in question  to an infinite cyclic cover of $M_0$;
  \item constructing a conjugacy map by carefully perturbing along flowlines.
\end{enumerate}
The lifting technique provide us a convenient property about the uniqueness of  intersectional points between two leaves in the two induced $1$-foliations on the glued torus.
Note that some similar techniques to the first point  were already used by Barbot \cite{Bar} and Beguin and the second author of this paper \cite{BY1} in other cases. We think that some further developments of these ideas will be important for classifying Anosov flows on toroidal $3$-manifolds.

\subsection{Generalizations}\label{ss.Ge}
Note that  the stable foliation of the attractor $\Lambda_A$  induces a $1$-foliation which is the union of two Reeb annuli on $\partial N_A$. One can choose a compact leaf $c_v$ in the induced $1$-foliation and a circle $c_h$ which bounds a
once-punctured torus in $N_A$ so that
$c_v$ and $c_h$ intersect once. We can fix an orientation on $c_v$ and an orientation on $c_h$
so that they can be used to
coordinate both  $\partial N_1$ and $\partial N_2$. Then an orientation preserving homeomorphism
$\Psi: \partial N_2 \to \partial N_1$ is always isotopic to an automorphism $B\in SL(2, \ZZ)$
so that $(B\circ c_h, B\circ c_v)=(c_h,c_v)B$.
We denote by $M_B^A$ the $3$-manifold obtained by gluing two copies of $N_A$, say $N_1^A$ and $N_2^A$,
  through a gluing automorphism $B$. In particular, when $A=\left(
                                    \begin{array}{cc}
                                      2 & 1 \\
                                      1 & 1 \\
                                    \end{array}
                                  \right)$, we replace $M_B^A$ by $M_B$.
                                  Obviously, if $B=\left(
                                    \begin{array}{cc}
                                      1 & 0 \\
                                      0 & 1 \\
                                    \end{array}
                                  \right)$, $M_B = M_0$.
When
$B=\left(
                                    \begin{array}{cc}
                                      1 & 0 \\
                                      k & 1 \\
                                    \end{array}
                                  \right)
  $ or
  $B=\left(
                                    \begin{array}{cc}
                                      -1 & 0 \\
                                      k & -1 \\
                                    \end{array}
                                  \right)
  $ for some $k\in \ZZ$,   similar to  the construction of $Z_t^0$ on $M_0$, one can construct a non-transitive Anosov flow $Z_t^{B}$ on $M_B^A$. More details can be found in Section \ref{s.final}.

It is natural to ask for the following   generalizations of Theorem \ref{t.claexp} and  Theorem  \ref{t.claAno}:
\begin{question}\label{q.gen}
\begin{enumerate}
  \item How to classify non-transitive Anosov flows on $M_B$?
  \item How to classify expanding attractors on $N_A$ and non-transitive Anosov flows on $M_B^A$?
\end{enumerate}
\end{question}

   We have the following theorem which answers the first question.

\begin{theorem}\label{t.gFW1}
\begin{enumerate}
  \item If $B= \left(
                                    \begin{array}{cc}
                                      1 & 0 \\
                                      k & 1 \\
                                    \end{array}
                                  \right)
  $  or
  $B=\left(
                                    \begin{array}{cc}
                                      -1 & 0 \\
                                      k & -1 \\
                                    \end{array}
                                  \right)
  $ for some $k\in \ZZ$, then every non-transitive Anosov flow on  $M_B$  is topologically equivalent to  $Z_t^B$;
  \item otherwise, $M_B$  does not carry any non-transitive Anosov flow.
\end{enumerate}
\end{theorem}

We remark that the proof of item $1$ of the theorem essentially is the same as the proof of Theorem \ref{t.claAno},
and  the proof of item $2$ basically depends on a transversality obstruction
for two special $1$-foliations on a torus.

Our discussion about the second  question  in Question \ref{q.gen} is more complicated. Nevertheless, except for
some very special cases, we also can answer the question similar to
Theorem \ref{t.claexp} and Theorem \ref{t.claAno}.  But the statements and comments have to be very subtle, so we
collect them in Section \ref{s.final}, in particular, Theorem \ref{t.claexpA}.

\subsection{Further questions}

The first natural question is:

\begin{question}\label{q.transitive}
Does there exist any transitive Anosov flow on $M_0$, $M_B$ or $M_B^A$?
If so, how to classify such flows?
\end{question}
 We remark that the question above is a special case of the following,
 more general question:
  classify transitive Anosov flows on toroidal $3$-manifolds
with two hyperbolic JSJ pieces. Some recent progresses for discussing similar questions on  graph manifolds can be found in \cite{BF1} and \cite{BF2}.

Unlike the phenomena implied in Theorem \ref{t.claexp},
for every $n\in \NN$,
there exists a hyperbolic $3$-manifold
$N_n$ which supports at least $n$ pairwise non-topologically equivalent expanding attractors, see \cite{BM} \footnote{In \cite{BM}, the authors show that for every $n\in \NN$, there exists a hyperbolic $3$-manifold
$M_n$ which carries at least $n$ pairwise non-topologically equivalent Anosov flows.
We remark that , by doing some DA surgeries in their examples, one can immediately get the similar result for expanding attractors}.  Beguin and the second author  of the paper  prove this result by a different method in a preparing work  \cite{BY2}.
 So, it is natural to ask:

\begin{question}\label{q.fiExp}
Does every compact orientable hyperbolic $3$-manifold with finitely many tori boundary only admit finitely many expanding attractors up to topological equivalence?
\end{question}

Note that the similar question for Anosov flows on $3$-manifolds is an  open question.

\section*{Acknowledgments}
This work started from a summer school at Peking University in 2018, organized by Shaobo Gan and Yi Shi.
We thank them for providing us the chance  to meet together.
This work was partially carried during some stay of Jiagang Yang in Tongji University and Bin Yu in Southern University of Science and Technology (SUSTech). We
thank these universities for the financial support for these visits.
In particular, Jiagang Yang was visiting SUSTech's Mathematics Department for a long-term academic leave when this paper was preparing. He is very grateful for the good working environment during his visit and for the support received from
the Department colleagues and authorities, in particular from Jana Rodriguez Hertz and Raul Ures.
We also would like to thank Sebastien Alvarez, Francois Beguin, Christian Bonatti, Youlin Li and Raul Ures   for their mathematical comments and suggestions. In particular, we thank Fan Yang for his numerous language corrections.
Jiagang Yang is supported by CNPq, FAPERJ, and PRONEX.
Bin Yu is supported by the National Natural Science Foundation of China (NSFC 11871374).

\section{Preliminaries}
\subsection{Expanding attractors}\label{ss.epa}
In this subsection, we will briefly introduce expanding attractors of diffeomorphisms on surfaces and
expanding attractors of flows on $3$-manifolds. The story about the study on these two topics are quite different:
there are  many  works from several different perspectives on the first topic, but there are  only a few tentative approaches
on the second.

\subsubsection{Expanding attractors of diffeomorphisms on surfaces}\label{sss.epadif}
The topological classification of expanding attractors on surfaces are well studied by several authors from different viewpoints
(see for instance, \cite{Ply},  \cite{Rua} and \cite{BLJ}).
In this section we adopt the viewpoint which exploits the relation
between expanding attractors and Pseudo-Anosov homeomorphisms.

Let $\Sigma_0$ be an orientable compact surface and $\Phi_0:\Sigma_0 \to \Sigma_0$ be  an embedding map so that:
\begin{enumerate}
  \item $\Phi_0 (\Sigma_0)$ is in the interior of $\Sigma_0$;
  \item the maximal invariant set of $\Phi_0$ on $\Sigma_0$ is an expanding attractor $\Lambda_a$, i.e. a transitive uniformally
  hyperbolic attractor so that the topological dimension of $\Lambda_a$ is $1$;
  \item every connected component of a stable manifold minus $\Lambda_a$ contains one end in $\Lambda_a$.
\end{enumerate}
We say that such a $\Sigma_0$ is a \emph{filtrating neighborhood} of $\Lambda_a$.
The union of the stable manifolds of $\Lambda_a$ induces a foliation $f^s$ on $\Sigma_0$.
A periodic orbit $\Gamma =\{P_1,\dots, P_k\}$ of $\Lambda_a$ is called a \emph{boundary periodic orbit}
if there exists a free separatrix of $W^s (P_i)$ ($i=1,\dots,k$). Each $P_i$ is called a \emph{boundary periodic point}.
A separatrix $W_+^s(P_i)$ of $W^s(P_i)$ is \emph{free} if there exists an open sub-arc of $W_+^s (P_i)$ with one end $P_i$ which
is disjoint to $\Lambda_a$.
We say that two boundary periodic points $P_1$ and $P_2$ are \emph{adjacent} if there are two unstable separatrices $W_+^u (P_1)$ and $W_-^u (P_2)$ of $P_1$ and $P_2$ so that they can be connected by a  stable segment whose interior is disjoint with $\Lambda_a$.
We say two boundary periodic orbits $P_1$ and $P_2$ are \emph{chain-adjacent} in the sense
that there exists some boundary periodic points $P^i$ ($i=1,\dots,n$) so that $P^1= P_1, P^n =P_2$ and $P^i$ and $P^{i+1}$ ($i=1,\dots,n-1$) are adjacent.

The following theorems are instrumental to understand expanding attractors on surfaces. They can be found in Ruas' thesis \cite{Rua} or Chapter $8$
of Bonatti-Langevin-Jeandenans \cite{BLJ} \footnote{Indeed, in Chapter $8$ of \cite{BLJ}, Bonatti and Jeandenans dealt with a different and more complicated class of hyperbolic basic sets, i.e,
saddle basic sets. But their idea also works in the case of expanding attractors.}

\begin{theorem} \label{t.boundsuf}
\begin{enumerate}
  \item There are finitely many boundary periodic orbits in $\Lambda_a$;
  \item the chain-adjacent relation is a equivalent relation among boundary periodic orbits;
  \item every chain-adjacent class $\{P^1, \dots, P^n \}$ is associated to a boundary connected component so that $P^i$ and $P^{i+1}$ ($i=1,\dots,n$) are adjacent where we assume that
      $P^{n+1}=P^1$
      (See Figure \ref{f.Cadjclass}).
\end{enumerate}
\end{theorem}

 \begin{figure}[htp]
\begin{center}
  \includegraphics[totalheight=5cm]{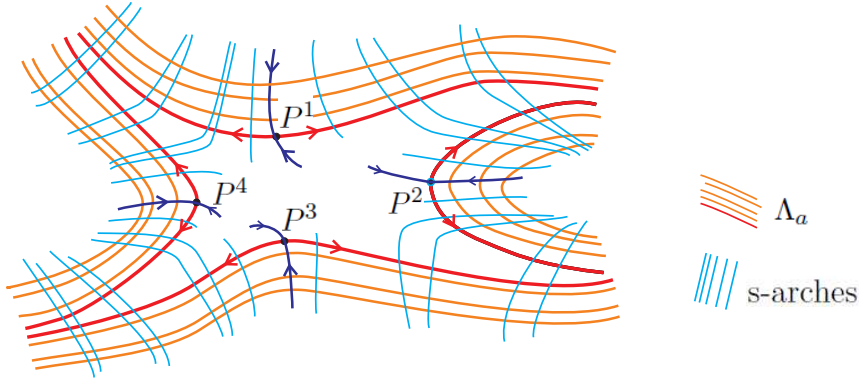}\\

  \caption{a chain-adjacent class with $4$ boundary periodic orbits}\label{f.Cadjclass}
\end{center}
\end{figure}

\begin{definition}\label{d.filA}
There exists a standard way to extend $\Phi_0$ on $\Sigma_0$ to an Axiom A diffeomorphism
$\overline{\Phi}_0$ on a closed surface $\Sigma$ so that:
\begin{enumerate}
  \item $\Sigma$ can be obtained  by filling disks to all of the boundary components of $\Sigma_0$;
  \item $\overline{\Phi}_0$ is an Axiom A diffeomorphism which is the extension of $\Phi_0$ so that each filling disk contains a unique repelling periodic point.
\end{enumerate}
It is easy to show that up to topological conjugacy, $\overline{\Phi}_0$ is unique. We call $ \overline{\Phi}_0$ the \emph{filling Axiom A diffeomorphism} of $\Phi_0$.
\end{definition}

We call an \emph{$s$-arch} for $\Phi_0$ every arc of stable manifold with both ends  in the attractor $\Lambda_a$, but whose interior is disjoint from $\Lambda_a$.

\begin{theorem}\label{t.sufconj}
There exists  a pseudo-Anosov homeomorphism $\Phi$ on $\Sigma$
and a continuous surjective map $\pi: \Sigma_0 \to \Sigma$  so that
$$\pi \circ \Phi_0= \Phi \circ \pi$$
and every preimage $\pi^{-1} (y)$ of $\pi$ at a point $y\in \Sigma$
is either an $s$-arch for $\Phi_0$
or the free stable separatrices of a chain-adjacent class of $\Phi_0$.
Moreover,  $\Phi$ and $\overline{\Phi}_0$ are isotopic.
\end{theorem}

\begin{remark}\label{r.sufconj}
\begin{enumerate}
  \item The semi-conjugate  map $\pi$  crashes every $s$-arch to a point, and
  also crashes the free separatrices according to the same chain-adjacent class of
  $\Phi_0$ to a point. We call this shrinking procedure an \emph{s-arches shrinking}.
  \item As a direct consequence of Theorem \ref{t.sufconj}, the preimage of a singular point of
  $\Phi$ is the union of the free stable separatrices of a chain-adjacent class of $\Phi_0$.
\end{enumerate}
\end{remark}

\subsubsection{Expanding attractors of flows on $3$-manifolds}\label{sss.epaflow}

Let $Y_t$ be a smooth flow on a compact orientable $3$-manifold $M$
and $N\subset M$ is a compact codimension-$0$ sub-manifold  with boundary so that
$Y_t$ is transverse inwards $N$ along $\partial N$ and
$\Lambda$ is the maximal invariant set
of $(N,Y_t \mid_N)$. If $\Lambda$ is a transitive uniformally hyperbolic attractor so that both  the topological dimension  and the unstable dimension
of $\Lambda$ are $2$, then we say that $\Lambda$ is an \emph{expanding attractor} and $N$ is a \emph{filtrating neighborhood} of $\Lambda$. Sometimes, we say that  $\Lambda$ is \emph{supported} by $N$.
Indeed, for $(M, Y_t, \Lambda)$, filtrating neighborhood of $\Lambda$ is unique up to topological equivalence  induced by the flow orbits.

The union of the stable manifolds of $\Lambda$ induces a foliation $\cF^s$ on $N$.
Since $Y_t$ is transverse inwards $N$ along $\partial N$, $\cF^s$ is transverse to $\partial N$ too.
Now we can define some similar notions  to the case of diffeomorphisms on surfaces.
A periodic $\gamma$ of $\Lambda$ is called a \emph{boundary periodic orbit} if
there is a free separatrix of $W^s (\gamma)$. We say a separatrix $W_+^s (\gamma)$ is \emph{free} if
there is an open neighborhood of $\gamma$ which is disjoint with $\Lambda$ except for $\gamma$.
We say that two boundary periodic orbits $\gamma_1$ and $\gamma_2$ are \emph{adjacent} if there are two unstable separatrices $W_+^u (\gamma_1)$ and $W_-^u (\gamma_2)$ of $\gamma_1$ and $\gamma_2$ so that they can be connected by a (strongly) stable segment whose interior is disjoint with $\Lambda$.
We say two boundary periodic orbits $\gamma_1$ and $\gamma_2$ are \emph{chain-adjacent} in the sense
that there exists some boundary periodic orbits $\gamma^i$ ($i=1,\dots,n$) so that $\gamma^1= \gamma_1, \gamma^n =\gamma_2$ and $\gamma^i$ and $\gamma^{i+1}$ ($i=1,\dots,n-1$) are adjacent.

The following theorem which is parallel to Theorem \ref{t.boundsuf} is the first important step to describe expanding attractors on $3$-manifolds,
which are essentially due to Christy \cite{Ch1}, Beguin-Bonatti \cite{BB}.\footnote{Indeed, \cite{BB}, Beguin and Bonatti studied a more complicated case: nontrivial $1$-dimensional saddle basic set. Nevertheless, their work also can imply Theorem \ref{t.boundpeorb}.}

\begin{theorem} \label{t.boundpeorb}
\begin{enumerate}
  \item There are finitely many boundary periodic orbits in $\Lambda$;
  \item the chain-adjacent relation is a equivalent relation among boundary periodic orbits;
  \item every  chain-adjacent class $\{\gamma^1, \dots, \gamma^n \}$ is associated to a boundary connected component so that $\gamma^i$ and $\gamma^{i+1}$ ($i=1,\dots,n$) are adjacent where we assume that
      $\gamma^{n+1}=\gamma^1$;
\end{enumerate}
\end{theorem}

Recall that $\cF^s$ is transverse to $\partial N$, then $F^s = \cF^s \cap \partial N$  is a $1$-foliation. By Poincare-Hopf theorem, every connected component of $\partial N$ is homeomorphic to either a torus or a Klein bottle.
Note that $M$ is orientable, then every connected component of $\partial N$ should be homeomorphic to a torus.
 Each compact leaf of $F^s$ is the intersection of the free separatrix of a boundary periodic orbit and $\partial N$. Moreover,  there always exist compact leaves in every connected component $T$ of $\partial N$ and
 the  number of the compact leaves of $F^s$ on $T$ is equivalent to the number of boundary periodic orbits which are associated to $T$. Then $F^s$ on $T$ is the union of finitely many Reeb
annuli and fundamental non-Reeb annuli. Here, a \emph{fundamental non-Reeb annulus} is an annulus endowed with a $1$-foliation so that,
  \begin{enumerate}
    \item the foliation  only contains two compact leaves whose union is the boundary of the annulus;
    \item  every non-compact leaf is asymptotic to the two compact leaves in different directions.
  \end{enumerate}
  Since $F^s = \cF^s \cap \partial N$, for every boundary compact leaf $c\in F^s$, the holonomy of $F^s$ along $c$ is contractable in the same directions. As a combinatorial consequence, the number of fundamental non-Reeb annuli of $f^s$ on $T$ is even.
We say the expanding attractor $\Lambda$ is \emph{coherent} if $F^s$ on every boundary torus is the union
of finitely many Reeb annuli. Equivalently and more intrinsically, every two adjacent boundary periodic orbits are co-oriented in the path closure of $N-\Lambda$.
Otherwise, $\Lambda$ is called an \emph{incoherent} attractor.

%\begin{proposition}\label{p.fol}
%\begin{enumerate}
 % \item For every boundary compact leaf $c\in f^s$, the holonomy of $f^s$ along $c$ is contractable in the same directions.
 % \item For every connected component of $\partial N$, the number of non-Reeb annuli is even.
 % \item $\cF^s$ has no compact leaf and the holonomy of every leaf (along the reverse orientation on periodic orbit??)
  %is expansive.
%\end{enumerate}
%\end{proposition}

%\begin{proposition}\label{p.bound}
%Let $\Lambda$ be the expanding attractor of the smooth vector field $Y$ on a compact $3$-manifold $N$ which is a filtrating neighborhood
%of $\Lambda$. Then, we have:
%\begin{enumerate}
 % \item there are finitely many boundary periodic orbits (ref?);
  %\item $\Lambda$ is a coherent attractor if and only if $f^s$ is the union of finitely many Reeb annuli;
%\end{enumerate}
%\end{proposition}

In the introduction, we have roughly introduced a class of expanding attractors which plays a fundamental role in this paper. Let us reconstruct them with more descriptions.
Starting with the suspension Anosov flow on the sol-manifold
$W_A=T^2 \times [0,1]/ (x,1)\sim (A(x),0)$ induced by  the vector field $(0, \frac{\partial}{\partial t})$, then we do a DA-surgery  on a small tubular neighborhood of the periodic orbit associated to the origin in $T^2$, one obtain a new Axiom A flow $Y_t^A$  on $W_A$  with one expanding attractor $\Lambda_A$ and one isolated repeller $\gamma_A$.
%There exists a smooth circle $c$ in $T^2$
%which cut $T^2$ to a once-punctured torus $\Sigma_0$ and a disk $D$ so that
%$\Sigma_0$ is a filtrating neighborhood of the corresponding $1$-dimensional expanding attractor $\Lambda_a$.
%The existence of such a $c$ can be followed by a regular level set of a Lyapunov function of $\Psi_A$,
%and also can be followed by Lemma \ref{l.ncircle}. In the same sprit,
Indeed, we can cut a  small tubular neighborhood $U(\gamma_A)$ of
$\gamma_A$ so that the the path closure of $W_A -U(\gamma_A)$, named by $N_A$, satisfies the following conditions:
 \begin{enumerate}
   \item it is a filtrating neighborhood  of $\Lambda_A$;
   \item $\Sigma_t= N_0 \cap T^2 \times \{t\}$  is a filtrating neighborhood of the attractor $\Lambda \cap \Sigma_t$ for the first return map of $Y_t$ on $\Sigma_t$.
 \end{enumerate}
Note that $\Sigma_t =\Sigma_{t+1}$ for every $t\in \RR$.
We collect some further properties about $(N_A, Y_t^A)$ here without proofs, since they can be easily checked using the descriptions above.

\begin{proposition}\label{p.pYt}
\begin{enumerate}
  \item $\{\Sigma_t\}$ induces a once-punctured torus fibration structure on $N_A$.
  \item We can define a projection map $P_1: N_A \to \SS^1$  by $P_1 (x,t)=t$ so that $P_1^{-1}(t)=\Sigma_t$.
  \item The stable foliation of $\Lambda_A$, $\cF^s$ is transverse to $\partial N_A$. This property implies that $\cF^s \cap \partial N_A$ is a $1$-foliation, named by $F^s$.
  \item $F^s$ is    composed of two Reeb annuli so that the holonomy of each compact leaf is either attracting or repelling, and  holonomy repelling direction of a compact leaf is coherent to the orientation of a boundary periodic orbit. (Figure \ref{f.DAatt})
  \item The maximal invariant set $\Lambda_A$ of $(N_A,Y_t^A)$  is a coherent expanding attractor.
\end{enumerate}
\end{proposition}

\begin{figure}[htp]
\begin{center}
  \includegraphics[totalheight=7.5cm]{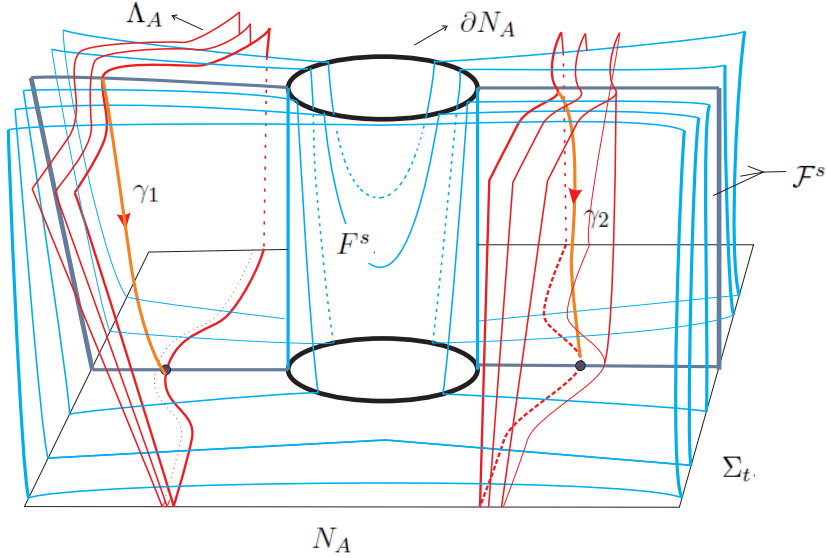}\\

  \caption{$(N_A, Y_t^A)$ with two boundary periodic orbits $\gamma_1$ and $\gamma_2$}\label{f.DAatt}
\end{center}
\end{figure}

\subsection{Transverse surface}
A transverse surfaces of
a $3$-dimensional Anosov flow share some very nice properties (see Proposition \ref{p.BrFe}). They were first proved by Brunella \cite{Br} and Fenley \cite{Fen2}
independently. Note that their strategies \footnote{for instance, the proof of Lemma 2 in \cite{Br}} can be routinely used to prove the same properties for expanding attractors of $3$-dimensional flows.
We collect them  as follows.

\begin{proposition}\label{p.BrFe}
Let $\Sigma$ be a transverse surface of an Anosov flow  on a closed orientable $3$-manifold  or
an expanding attractor  supported by a compact orientable $3$-manifold.
Then we have:
\begin{enumerate}
  \item $\Sigma$ is homeomorphic to a torus;
  \item if $T_1$ and $T_2$ are two transverse tori, then
they can be isotopic along the flowlines of the corresponding flow.
\end{enumerate}
\end{proposition}

\subsection{Several strategies to construct foliations}
In this subsection, we will introduce several strategies which will allow us to construct
new foliations from old one. 
Most parts in this subsection  can be found in  Chapter $4$ of Calegari's book \cite{Cal} or Gabai \cite{Ga}.

\subsubsection{Filling Reeb annulus}
Suppose $N$ is a  compact $3$-manifold with boundary and $\cF$
is a $2$-foliation of $N$ so that there is an annulus $A\subset \partial N$
satisfying that $\cF\cap A$ is a Reeb annulus. Then one can glue a solid torus $V$
with half-Reeb foliations (see Figure \ref{f.halfreeb} as an illustration) to $N$ along $A$, so that
the two foliations are pinched to a foliation or a branching foliation
 on the  glued manifold $N'\cong N$. Note that after this surgery, compare to the foliation on $N$,
 there is one more tangent annulus in the boundary of $N'$.

\begin{figure}[htp]
\begin{center}
  \includegraphics[totalheight=6.3cm]{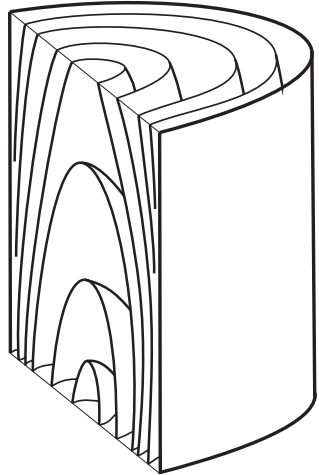}\\

  \caption{half Reeb foliations in a solid torus (the top and the bottom  are glued together)}\label{f.halfreeb}
\end{center}
\end{figure}

\subsubsection{Blowing up leaves and filling gaps}\label{sss.blup}
Suppose $N$ is a compact $3$-manifold with boundary and $\cF$
is a $2$-foliation on $N$. Assume that $l$ is a noncompact leaf of $\cF$ so that $l$ is transverse to $\partial N$ and there exists a circle $c \subset l \cap \partial N$. We open $l$ to $N(l)$ where $N(l)$ is an $I$-bundle over $l$ and $\partial N(l)=l_1 \cup A_0 \cup l_2$, where $A_0$ is an annulus with boundary the union of two circles $c_1$ and $c_2$.
 Then  $\cF$ is opened along $l$ to a lamination $\cL$ \footnote{A reader can read Example 4.14 of \cite{Cal} to find a serious explanation about why $\cL$ can be a lamination.} so that the path closure of $N-\cL$ is $N(l)$.

Let $F$ be a $1$- foliation on $A_0$ which is transverse to the $I$-bundle on $A_0$. For instance, one can choose $(A_0, F)$ to be a fundamental non-Reeb annulus.
We will explain that  $F$ always can be extended to a $2$-foliation on $N(l)$  transverse to the $I$-bundle.
Naturally,
the  extended $2$-foliation and the lamination $\Lambda$ are pinched to a $2$-foliation on $N$.

Since  $F$ is transverse to the $I$-bundle on $A_0$,  $F$ can be described by a holonomy homeomorphism $h$ on $[0,1]$ along $c$ so that $h(0)=0$ and $h(1)=1$.
Note that a $2$-foliation $\cF^2$ in $N(l)$ transverse to the I-bundle
can be described by a holonomy homomorphism:
$$\varphi: \pi_1(l)\to Homeo([0,1])$$
Therefore, to extend $F$  to a $2$-foliation on $N(l)$ transverse to the $I$-bundle
is equivalent to build a holonomy homomorphism $\varphi:\pi_1 (l) \to Homeo ([0,1])$ so that $\varphi([c])=h$.
But this always can be done since $l$ is non-compact, and  then $[c]$ is a generator of $\pi_1 (l)$, which is isomorphic to a free group. One also see a similar argument in Example 4.22 of \cite{Cal}.

%\subsubsection{pushing extensions}
%Suppose that $W$ is a compact $3$-manifold and $T$ is a connected component of $\partial W$.
%Assume that $\cF$ is a $2$-foliation on $W$. Moreover, assume that there are two disjoint annuli $A_1$ and $A_2$ in $T$
%so that $\cF$ is transverse to $A_1$ and $A_2$ so that up to leaf-conjugacy,  $F_1=\cF \cap A_1$ is the mirror image of $F_2 = \cF \cap A_2$ by a reflection
%along a simple closed curve isotopic to $\partial A_1$ in $T$.
%For instance, $F_1=\cF \cap A_1$ and $F_2 = \cF \cap A_2$ are two Reeb annuli with the same asymptotic orientations,
%or two fundamental non-Reeb annuli which are mirror images each other.

%Now we can find a solid torus $V=S^1 \times [0,1] \times [0,1]$ filling with foliation $F\times [0,1]$ so that $F$ is leaf-conjugate to
%$F_i$ ($i=1,2$). Now we can glue $V$ to $W$ by gluing homeomorphisms $\psi_1: S^1 \times [0,1] \times \{0\} \to A_1$ and
%$\psi_2: S^1 \times [0,1] \times \{1\} \to A_2$ so that both of them preserve the corresponding $1$-foliations.
%Therefore, we get a new foliation $\cF'$ on a new manifold
%$M=W\cup V$.  We define by a \emph{pushing extension} the surgery from $(W, \cF)$ to $(M, \cF')$.

\subsubsection{Surgeries with saddles}
We are going to introduce monkey saddle foliations on a solid torus, mainly from
Example 4.18, Example 4.19 and Example 4.22 of \cite{Cal}.

 We start at
$V=A_0\times [0,1]=S^1 \times [0,1] \times [0,1]$ with foliation $\cF_1= F_1 \times [0,1]$ where $F_1$ is a fundamental non-Reeb annulus on $A_0$. Recall that a fundamental non-Reeb annulus is an annulus endowed with
a $1$-foliation so that the foliation is non-Reeb, and there are exactly two compact leaves in the foliation
whose union is the boundary of the annulus.
Let $c_0=S^1 \times (0,0)$. Since $F_1$ is a fundamental non-Reeb annulus, we can assume that $c_0$ is transverse to $\cF_1$.
$\partial V$ is the union of four annuli $A_1, B_1, A_2, B_2$ where $A_1$ and $A_2$ are two  annuli
so that each of them is endowed with a fundamental non-Reeb annulus and $B_1$ and $B_2$ are two annuli leaves of $\cF_1$.
We call $\cF_1$ an \emph{index-$1$ monkey saddle foliation} on the solid torus $V$.

%Choose a circle $C$ in the center of $A\times [0,1]$ so that $C$ is transverse to  $\cF_1$ and
  % $C$ is parallel to the core of $A\times \{0\}$. We can think that $C$ is the core of $A\times \{\frac{1}{2}\}$.
   For every even number $p>2$, by doing a $\frac{p}{2}$-branched cover on $V=A_0\times [0,1]$ along $C$ and lifting $\cF_1$ under this branched cover,
   we get a new foliation $\cF_{\frac{p}{2}}$ on a new solid torus
   which is called an \emph{index-$\frac{p}{2}$  monkey saddle foliation}.
    For simplicity, we still call by $V$ the new
   solid torus. Now $\partial V$ is the union of $2p$ annuli $A_1, B_1, A_2, B_2, \dots, A_p, B_p$ where each $A_i$ is endowed with a fundamental non-Reeb annulus  and each $B_j$  is an annulus leaf of $\cF_{\frac{p}{2}}$.
Indeed, $\cF_3$ is the canonical monkey saddle foliation. (Figure \ref{f.monkeysaddle})

 \begin{figure}[htp]
\begin{center}
  \includegraphics[totalheight=5.5cm]{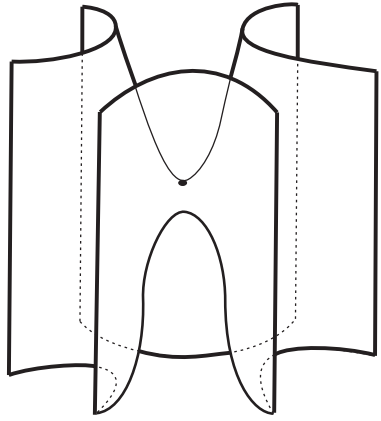}\\

  \caption{a leaf of monkey saddle foliation}\label{f.monkeysaddle}
\end{center}
\end{figure}

%Note that for index-$\frac{p}{2}$ monkey saddle foliation on $V$, the core of a leaf $B_p^j$ ($j\in \{1,\dots, p\}$)
%is a longitude of $V$. Indeed, we can build similar foliation on $V$ without this restriction.
%If we cut $(V, \cF_{\frac{p}{2}})$ along a $2p$ horozental  polygon to two polygons, then rotate one of them $\frac{2s\pi}{p}$,
%finally glue them together. Then, we get 'twist' monkey saddle foliation on $V$ so that the core of
%a boundary annulus maybe not a longitude of $V$. For simplicity, we still call it by an index-$\frac{p}{2}$ monkey saddle foliation.

%For convenient below, we uniform the parameters  of monkey saddle foliations in all cases.
%Let $c$ be an essential simple closed curve in $T=\partial V$ so that the absolute value of the algebraic intersection number of
%$c$ and $m$ is $q$ ($q>0$).
%For every integer $p>0$ so that $pq$ is even, we can build an index-$\frac{pq}{2}$ monkey saddle foliation $\cF_{\frac{pq}{2}}^c$
%so that there are $p$ boundary annuli leaves and the core of every boundary annulus is isotopic to $c$ in $T$.

\subsection{Boundary branching foliation and filling monkey saddle foliations}
Now we define a type of pseudo-foliation on a compact orientable $3$-manifold $N$ with a torus boundary $T$.
We say that $\cF_b$ a \emph{boundary  branching foliation}  if $T$ is the union of $p$ annuli $B^1, \dots, B^p$
so that
\begin{enumerate}
  \item $\partial B^i = a_{i-1} \cup a_i$ ($i=1,\dots,p$ and $a_0= a_p$) and $B^i \cap B^{i+1}=a_i$ where $a_i$ is  a simple closed curve in $T$;
  \item $T$ (the union of all of $B^i$) is in the unique  branching leaf $l_b$ of $\cF_b$ and $a_i$ is  a branching cusp circle;
  \item  the path closure of
$l_b - T$ is the union of $p$  immerse surfaces $l_h^1, \dots, l_h^p$, which are called by \emph{half-leaves} of $\cF_b$.
\end{enumerate}
(see Figure \ref{f.branchfol})
We say a boundary  branching foliation $\cF_b$ is \emph{special} if both of every half-leaf of $\cF_b$ and every other leaf of $\cF_b$
 are non-compact.
Similar to open boundary leaves, we can open every half-leaf in $l_b - T$ to get a lamination $\cL_{\cF_b}$ in $N$.
 Certainly, $\cF_b$ and $\cL_{\cF_b}$ are closely related, for instance:
 \begin{enumerate}
   \item $l_b$ is split to $p$ leaves without branching;
   \item every other leaf of $\cF_b$ is still corresponding to a leaf in $\cL_{\cF_b}$.
 \end{enumerate}
 The half-leaf $l_h^i$ ($i\in \{1,\dots,p\}$) is opened to $N(l_h^i)$ where $N(l_h^i)$ is an $I$-bundle over $l_h^i$ and $\partial N(l_h^i)=l_h^{i,+} \cup A^i \cup l_h^{i,-}$ where $A^i$ is an annulus with boundary the union of two circles $a_i^+$ and $a_i^-$ which is the splitting of  $a_i$.
 Now $T$ becomes the union of $2p$ annuli $A^1, B^1, A^2, B^2, \dots, A^p, B^p$.

 \begin{figure}[htp]
\begin{center}
  \includegraphics[totalheight=6.5cm]{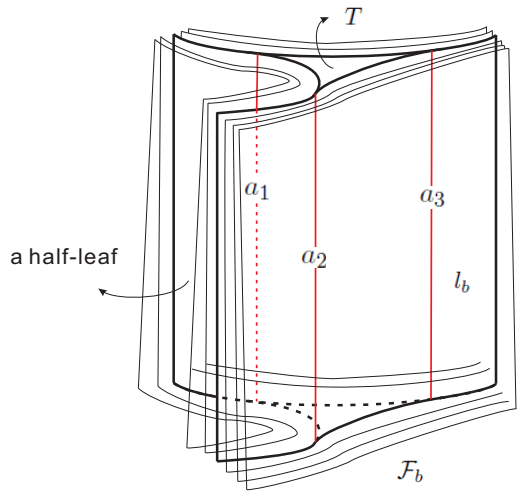}\\

  \caption{boundary  branching foliation $\cF_b$ close to $T$}\label{f.branchfol}
\end{center}
\end{figure}

Let $c$ be a circle in $T$ which intersects with $c_1$ once. Set $V$ is a solid torus with a meridian $c_m$ which bounds a disk in $V$.
We glue $N$ and $V$ together by a gluing map $\varphi: \partial V \to \partial N$ so that $c=c_m$ and we denote by $W$ the glued manifold $N\cup_T V$.
The following is  the main proposition of this subsection.

 \begin{proposition}\label{p.extfol}
 Let  $\cF_b$ be a special boundary branching foliation on $N$ with even number ($p$) of cusp circles  and $\cL_{\cF_b}$ be
 the lamination on $N$ by opening all of the half-leaves of $\cF_b$.
 Then there exists a taut foliation $\cF'$ on $W$ so that  $\cL_{\cF_b}$ is
 a sub-foliation of $\cF'$.
 \end{proposition}
 \begin{proof}
 Since $p$ is  even, then by the constructions of monkey saddle foliations introduced before, we can fill an index-$\frac{p}{2}$
 monkey saddle foliation $\cF_{\frac{p}{2}}$ in $V$ so that
  $A_1,  A_2, \dots, A_p$ are $p$ annuli
which are endowed with $p$ fundamental non-Reeb annuli and $B_1,  B_2, \dots, B_p$ are $p$ annuli components of $\cF_{\frac{p}{2}}$.
We can control the gluing map $\varphi$ between $N$ and $V$ carefully so that $\varphi(A_i)=A^i$ and $\varphi(B_i)=B^i$.

Since $B_i$ ($i=1,\dots,p$) is part of a leaf of   $\cL_{\cF_b}$
and $B^i$ also is a leaf of $\cF_{\frac{p}{2}}$,
then
$\cL_{\cF_b}$ and $\cF_{\frac{p}{2}}$ can be glued together.
To extend them to a foliation on $W$, we need to extend them to the interior of every $N(l_h^i)$.
Note that in $\partial N(l_h^i)$, each of $l_h^{i,+}$ and  $l_h^{i,-}$ is part of a leaf in $\cL_{\cF_b}$
and $F^i= A_i \cap \cF_{\frac{p}{2}}$ is a fundamental non-Reeb annulus which is a suspension foliation on
$A_i$. Also notice that $N(l_h^i)$ is an $I$-bundle over a non-compact surface $l_h^i$,
Then one can use the same technique in Section \ref{sss.blup} to extend the union of  $\cL_{\cF_b}$ and $\cF_{\frac{p}{2}}$
to a foliation $\cF'$ on $W$ so that $\cF'$ restricted to $N(l_h^i)$ is transverse to the $I$-bundle.

Obviously, $\cL_{\cF_b}$ is a sub-foliation of $\cF'$. We are left to check that $\cF'$ is a taut foliation. We only need to show that every leaf of $\cF'$ is non-compact.
First recall that since $\cF_b$ is special, then
$l_b - T$ is the union of $p$ half-leaves and both of every half-leaf of $\cF_b$ and every another leaf of $\cF_b$
 are non-compact. Therefore, every leaf of $\cL_{\cF_b}$ is non-compact.
 Let $l$ be a leaf in $\cF'$ which contains a monkey saddle leaf. One can observe that $l$ transversely intersects with $A_1$ so that $l\cap A_1$ is non-compact. Therefore, $l$ is non-compact.
 In summary, every leaf of $\cF'$ is non-compact.
 \end{proof}

\section{Filling foliations}\label{s.filfol}
Recall that $N_0$ is the compact $3$-manifold which is the path closure of the figure eight knot complement space,
$N_0$ also can be obtained by cutting an open solid torus from the sol-manifold $W_0$.
In this section we
assume that $Y_t$ is a smooth flow  on $N_0$ so that
the maximal invariant set of $(N_0,Y_t)$ is a hyperbolic expanding attractor $\Lambda$. Suppose that $\cF^s$ is the stable foliation and $F^s=\cF^s \cap T$  contains $p$ compact leaves $c_1, \dots, c_p$
            so that the holonomy of every $c_i$ (with a suitable orientation) is contracting.
           Here we denote by $T$ the torus $\partial N_0$.
\begin{lemma}\label{l.tranfol}
There exists a $1$-foliation $H^u$ on $T$  which is the union of $p$ Reeb annuli and is transverse to $F^s$ on $T$.
\end{lemma}
\begin{proof}
Firstly, we construct $H^u$.  Set    the path closure of $T-c_1 \cup \dots \cup c_p$ is the union of $p$ annuli $A_1, \dots, A_p$
so that $\partial A_i =c_i \cup c_{i+1}$ ($c_{p+1}=c_1$).
Then every $F^s \mid_{A_i}$ ($i=1, \dots, p$) either is a Reeb annulus or a fundamental non-Reeb annulus.
In any case, we can choose a circle $a_i$ in the interior of $A_i$ so that
$a_i$ is transverse to $F^s$.  Suppose $a_i$ cuts $A_i$ into $A_i^+$ and $A_i^-$.
For simplicity, we set $B=A_i^+$ or $B=A_i^-$. Then $F^s \mid_B$  can be   coordinated by $[0,1]\times \RR^1$ by standard $\ZZ$-action so that $F^s\mid_B$
 is induced by the orbits of the vector field $\cos (\frac{1}{2}\pi x) \frac{\partial}{\partial x}+\sin (\frac{1}{2}\pi x)\frac{\partial}{\partial y}$
  under the $\ZZ$-action. Now we can build $H^u$ on $B$ so that $H^u \mid_B$ is induced by the orbits of the vector field
  $\sin (\frac{1}{2}\pi x) \frac{\partial}{\partial x}-\cos (\frac{1}{2}\pi x) \frac{\partial}{\partial y}$ under the $\ZZ$-action.

  Note that $H^u \mid_B$ is the union of $a_i$ and infinitely many non-compact leaves so that each non-compact leaf is asymptotic to
  $a_i$ and vertical to either $c_i$ or $c_{i+1}$.   Therefore, they can be glued together to form a foliation $H^u$ on $T$ which is transverse to
  $F^s$. Moreover, notice that the holonomy of $F^s$ along every $c_i$ (with a suitable orientation) is contracting.
  One can automatically check that $H^u$ is the union of $p$ Reeb annuli (see a concrete example in Figure \ref{f.FuFs}).
\end{proof}

\begin{remark}\label{r.nattho}
We remind readers that as illustrated in Figure \ref{f.FuFs}, the holonomy of a compact leaf in $H$ is not necessary to be  attracting or repelling.
\end{remark}

\begin{figure}[htp]
\begin{center}
  \includegraphics[totalheight=6cm]{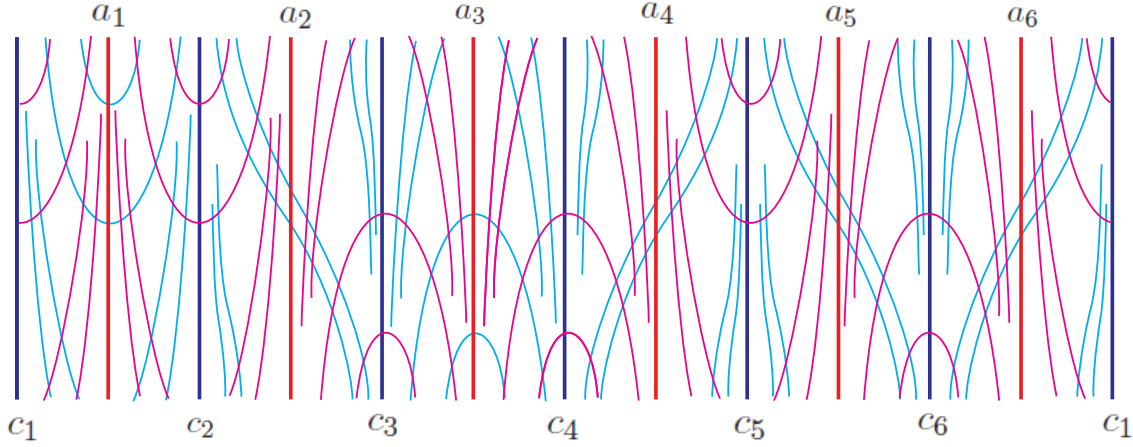}\\

  \caption{an example about $F^s$ (blue) and $H^u$ (red)}\label{f.FuFs}
\end{center}
\end{figure}

\begin{lemma}\label{l.uext}
Let $H^u$ be a $1$-foliation on $T$ which is transverse to $F^s$ and $\cL^u$ be the $2$-foliation
 induced by  the flowlines of the $1$-leafs of $H^u$  on $N_0-\Lambda$,
then
$\Lambda$ and $\cL^u$ form a $2$-foliation $\cF$ without any compact leaf on $N_0$.
\end{lemma}
\begin{proof}
This is a direct consequence of standard $\lambda$-lemma (\cite{PW}).
\end{proof}

 By Lemma \ref{l.tranfol} and Lemma \ref{l.uext}, we immediately have the following proposition.
\begin{proposition}\label{p.uext}
$\Lambda$ can  be extended to a foliation $\cH^u$  without any compact leaf on $N_0$ which is transverse to $T$ so that the induced $1$-foliation $\cH^u \cap T$ on $T$ is leaf-conjugate to $H^u$ in Lemma \ref{l.tranfol}.
\end{proposition}

Since $H^u = \cH^u \cap T$ is the union of $p$ Reeb annuli,
we can fill the $p$ Reeb annuli to get  $\cH_b^u$ on a compact $3$-manifold $N'$ which is homeomorphic to
$N_0$. For simplicity, we still call $N'$ by $N_0$. $\cH_b^u$ has the following property.

\begin{proposition}\label{p.pbrfol}
  $\cH_b^u$ is a    special boundary branching foliation on $N_0$ with $p$ branching $a_1, \dots, a_p$.   Moreover, the attractor $\Lambda$ is a sub-foliation
  of $\cH_b^u$.
\end{proposition}
\begin{proof}
By Lemma \ref{l.uext}   and Proposition \ref{p.uext},   every leaf of $\cH^u$ is non-compact and
 $a_i$ is contained in a non-compact leaf $l_i$ of $\cH^u$ which is homeomorphic to $S^1 \times [0,+\infty)$.
 Then one can automatically get that $\cH_b^u$ is a special    boundary branching foliation on $N_0$ with $p$ branching $a_1, \dots, a_p$.

 Moreover, notice that    $\Lambda$ is a sub-foliation    of $\cH^u$    which is disjoint with $T$. Therefore, $\Lambda$ is a sub-foliation  of $\cH_b^u$.
 \end{proof}

Recall that (see Section \ref{s.int}) by filling a solid torus $V$ to $N_0$ by gluing $T$ and $\partial V$  so that
the circle $c_h$ bounds a disk $D$
in $V$, we get the
Sol manifold $W_0$ which is homeomorphic to the mapping torus of $T^2$ under the
Thom-Anosov automorphism $A=\left(
                                    \begin{array}{cc}
                                      2 & 1 \\
                                      1 & 1 \\
                                    \end{array}
                                  \right)$. The purpose of this section is to show the following theorem.
% and $V$ is homeomorphic to a tubular neighborhood  of the periodic orbit $\gamma$  associated to the
% origin $O$ in the suspension Anosov vector field.

\begin{theorem}\label{t.folext}
Let $\Lambda$ be an expanding attractor with $N_0$ as a filtrating neighborhood,
then there exists an integer $q$  so that $\Lambda_q$ always
can be extended to a taut foliation $\cF_q$ on $W_q$.
Here $(W_q, \Lambda_q)$ is the $q$-cyclic cover of
$(W_0,\Lambda)$ with respect to the torus fibration structure on $W_0$.
\end{theorem}

Assume that the algebraic intersection number between $a_1$ and $c_h$ is $q$. We will divide the proof into the following three cases:
$pq$ is even and nonzero, $pq$ is odd and $pq=0$ (in fact, in this case $q=0$).

\subsection{$pq\neq 0$ and $pq$ is even}
 We can open every half-leaf of $\cH_b^u$ in $l_b - T$ so that we get a lamination $\cL_{\cH_b^u}$ in $N_0$
 with $\Lambda$ as  a sub-lamination.
 Since $pq$ is even and nonzero, we can do 'q'-cyclic cover with respect to the torus fibration structure.
 Let $(W_q, \Lambda_q, \cL_{\cH_b^u}^q)$ be the lift of $(W_0,\Lambda,\cL_{\cH_b^u})$. Note that the union of  compact leaves of $H^u$
 is lifted to the union of  $pq$ compact leaves so that each of them intersects $c_h^q$ once. Here $c_h^q$ is a lifted connected component of $c_h$ in
 $\partial W_q$.
 Then by Proposition \ref{p.extfol}, there exists a taut foliation $\cF_q$ on $W_q$
 so that  $\cL_{\cH_b^u}^q$ is a sub-foliation of $\cF_q$.   Since $\Lambda_q$ is a sub-lamination of $\cL_{\cH_b^u}^q$, then  $\Lambda_q$ also is a sub-lamination of
 the taut foliation $\cF_q$ on $W_q$. The conclusion of Theorem   \ref{t.folext} is followed in this case.

\subsection{$pq$ is odd}
We will show that this case can not appear.

 \begin{figure}[htp]
\begin{center}
  \includegraphics[totalheight=7.5cm]{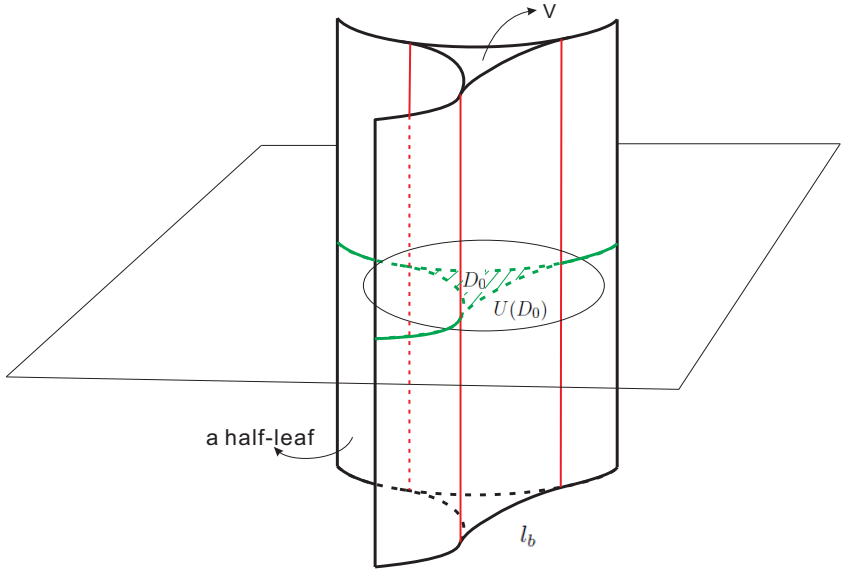}\\

  \caption{the local section $U(D_0)$}\label{f.localsec}
\end{center}
\end{figure}

\begin{proposition}\label{p.noodd}
For the parameters $p$ and $q$ of $\cH^u$, they can't satisfy that $pq$ is odd.
\end{proposition}
\begin{proof}
 We assume that $p$ and $q$ of $\cH^u$ satisfy that  $pq$ is odd. Set $\Sigma =T^2 \times \{0\}$ is an embedded torus in $W_0$.                                                                                                         Note that up to isotopy,  the core of $V$ can intersect $\Sigma$ once. Then up to isotopy we can assume that
$V$ transversely intersects with $\Sigma$, so that $V\cap \Sigma$ is a $pq$-polygon $D_0$, whose $pq$ edges are in the boundary of the $pq$ filling half-Reebs.
There exists a small disk neighborhood $U(D_0)$ of $D_0$
 in $\Sigma$ such that $\cH_b^u$ is transverse to $D(P_0)$ and the intersectional
 $1$-foliation is a local $pq$-branching singular foliation. (see Figure \ref{f.localsec} as an illustration) The proof of Lemma  $5$ in section $2$ of  Solodov \cite{So}  implied a more general conclusion:   let $S$ be an embedded surface in a $3$-manifold with $2$-foliation $\cF$ so that $\cF$ is transverse to $\partial S$, then $S$ can be isotopic to $S'$ relative to $\partial S$ so that $f_{sing}=\cF'\cap S'$ is a singular $1$-foliation with finitely many singularities so that everyone is either saddle or center type. \footnote{Indeed, Solodov (Lemma $5$, section $2$, \cite{So}) proved the claim in the case that $S$ is a disk. But one can routinely check that the technique which the author used doesn't depend on the topology of $S'$.}
In our case, by the conclusion above, up to isotopical leaf-conjugacy, one can assume that $\Sigma\setminus U(D_0)$ and $\cH_b^u$
are in a topologically general position:
\begin{enumerate}
\item $\Sigma\setminus U(D_0)$ and $\cH_b^u$  are transverse everywhere except for finitely many singularities;
 \item every singularity is either a saddle point or a center point.
 \end{enumerate}
  By blowing down the $pq$-polygon $D_0$ to a point $P_0$, the intersection $F_b^u = \cH_b^u \cap \Sigma$ endows a singular foliation on the torus
  $\Sigma$ whose singularity set is the union of one $pq$-branching singularity $P_0$ and finitely many singularities $P_1,\dots,P_k$ which is either a saddle point or a
center point. We set $P_i$ is an $2m_i$-prong singularity, then $m_i=0$ or $m_i=2$.
By 'general Poincare-Hopf theorem' for singular 1-foliations (see for instance, Page 75, \cite{FLP}),
$$\cX(\Sigma)= (1-\frac{pq}{2}) + \sum_{i=1}^k (1-m_i)$$
Note that $\cX(\Sigma)=0$ and  $pq$ is an odd number, then the equality isn't possible.  Therefore, $pq$ can't be odd.
\end{proof}
\subsection{$pq=0$}
  We will show that this case also does not exist. We remind readers that the definition of $c_h$ and $c_v$ can be found in Section \ref{s.int}.
 \begin{proposition}\label{p.no0}
 $a_1$ of $\cH^u$ is not parallel to $c_h$ in $T$. Equivalently, $pq\neq 0$.
 \end{proposition}
 \begin{proof}
  By the definitions of the notations, one can easily check that the following three statements are pairwise equivalent:
 $pq=0$, $q=0$ and  $a_1$ is parallel to $c_h$ in $T$.
  Therefore, we only need to show that $a_1$ is not parallel to
 $c_h$ in $T$.    We assume that $a_1$ is  parallel to $c_h$ in $T$.

 $N_0$ is  homeomorphic to the path closure of figure eight knot complement space with $c_v$ as a meridian circle.
 Recall that $N_0$ carries the special boundary branching foliation $\cH_b^u$ in $N_0$ with $p$ branching $a_1, \dots, a_p$.
 We can open every half-leaf of $\cH_b^u$ in $l_b - T$ so that we get a lamination $\cL_{\cH_b^u}$ in $N_0$
 with $\Lambda$ as  a sub-lamination.
 By filling a solid torus $V$ to $N_0$ so that a meridian of $V$ is glued to $c_v$, we get a glued manifold $M$ homeomorphic to $S^3$.
 If $p$ is even,
  by Proposition \ref{p.extfol}, there exists a taut foliation $\cF'$ on $M$
 so that  $\cL_{\cH_b^u}$ is a sub-foliation of $\cF'$. But $M$ is  homeomorphic to $S^3$ which does not carry any taut foliation.
 Therefore,  $p$ can not be an even number in this case.
 If $p$ is odd, we can do a branched double cover on $M$ along a core circle $K$ of $V$ which is a figure eight knot in $M\cong S^3$.
 We get a new manifold $\overline{M}$ which is homeomorphic to
 a lens space $L(5,2)$ (see Section C of Chapter $10$ in \cite{Ro}). $N_0, \cH_b^u, H^u, V$ are lifted to $\overline{N}_0, \overline{\cH}_b^u, \overline{H}^u, \overline{V}$ respectively.
 Then $\overline{\cH}_b^u$ is a special boundary branching foliation with $2p$ branchings so that each of them bounds a disk in $\overline{V}$.
 Then similar to the case $p$ is even, one can build a taut foliation on $\overline{M}$ from the boundary branching foliation $\overline{\cH}_b^u$.
 But $\overline{M}$ is a lens space whose universal cover is $S^3$, then $\overline{M}$ also does not carry any taut foliation.  We get a contradiction.

 In summary, $a_1$ can not be parallel to $c_h$ in $T$.
\end{proof}

\section{expanding attractors on the figure eight knot complement space $N_0$}\label{s.claexp}
In this section, we will show that $N_0$ carries a unique expanding attractor, i.e. Theorem \ref{t.claexp}.
We sketch the proof as follows.
\begin{enumerate}
  \item The flow in question on $N_0$ can be naturally extended to an Axiom A flow on $W_0$.
   we will show that every flowline of the Axiom flow essentially intersects with every torus fiber in $W_0$
       (Proposition \ref{p.liminf}). See Section \ref{ss.struflow}.
  \item Using on a beautiful theorem of Fuller \cite{Fu} (Theorem \ref{t.Fuller}), we find a global section for the Axiom A flow (Proposition \ref{p.sect}). 
      The existence of a global section allows us to translate Theorem \ref{t.claexp} into 
      the classification of a special class of Axiom A diffeomorphisms on a torus. 
      See Section \ref{ss.glsec}.
  \item Finally we prove a classification result for this class of  Axiom A diffeomorphisms (Proposition \ref{p.DAdiff}). See Section \ref{ss.claAdiff}. Then we can quickly finish the proof of Theorem \ref{t.claexp}.  See Section \ref{ss.fclaexp}.
\end{enumerate}

Now we introduce some notations and facts which will be used in this section.
Recall that $N_0$ is a sub-manifold of the sol-manifold $W_0$.
Let $Y_t$ be a smooth flow on $N_0$ so that the maximal invariant set
is a transverse expanding attractor $\Lambda$ supported by $N_0$.
The closure of $W_0 - N_0$, denoted by $V_0$, is homeomorphic to a solid torus,
therefore, up to topological equivalence, $Y_t$ can be extended to an Axiom A flow
on $W_0$ whose non-wandering set is the union of
the expanding attractor $\Lambda$ and an isolated periodic orbit repeller
$\beta$ which satisfies the following conditions:
 \begin{enumerate}
   \item $\beta$ is isotopic to the core of $V$;
   \item $\beta$ is transverse to the torus fibration
structure induced by $\{T^2 \times \{t\} \}$.
 \end{enumerate}
  Here the corresponding coordinate for the fibration structure can be found in Section \ref{sss.epaflow}. For simplicity, we still denote by $Y_t$ the extended Axiom A flow.

Also recall that $W_q$ is the $q$-cyclic cover of $W_0$ with respect to the torus fibration on $W_0$.
We remark that  $W_1=W_0$. Let $\pi_q: W_q \to W_0$ be the corresponding covering map.
Let $N_q$, $\Lambda_q$, $\beta_q$ and $Y_t^q$ be the  lifts of $N_0$, $\Lambda$, $\beta$ and $Y_t$ in $W_q$ under $P_q$. By the construction of $\pi_q$, $\beta_q$  is an isolated periodic orbit repeller
of $Y_t^q$.
Let  $\Wi W_0$ be the infinite cyclic covering space of $W_0$ associated to the torus fibration structure.
 $\Wi W_0$ is homeomorphic to $\Sigma \times \RR = \Sigma \times (-\infty, +\infty)$ where $\Sigma$ is homeomorphic to a torus.
 Assume that $\pi: \Wi W_0 \to W_0$ is the natural covering map and $P: \Wi W_0 \to \RR$ is the natural projection map.
 Set $\Wi N_0$,  $\Wi{\Lambda}$, $\Wi \beta$ and $\Wi Y_t$ are the lifts of $N_0$,  $\Lambda$, $\beta$ and $Y_t$ in $\Wi W_0$ respectively under  $\pi$.
 By the construction of $\pi$, $\Wi \beta$  is homeomorphic to $\RR$ and transverse to the torus fibration structure induced by $\{\Sigma \times \{t\}\}$.
Naturally $\Wi W_0$ also is the infinite cyclic cover of $W_q$ by a covering map $\pi^q: \Wi W_0 \to W_q$ so that $\pi_q \circ \pi^q =\pi$.
By Theorem \ref{t.folext},
we can choose an integer $q$  so that $\Lambda_q$
can be extended to a taut foliation $\cF_q$ on $W_q$.
We fix this $q$ from now on and
 further assume that $\Wi{\cF}$ is the lifting foliation of $\cF_q$ on $\Wi W_0$ under $P$.
 Certainly $\Wi \Lambda$ is a sub-foliation of $\Wi F$.

\subsection{The structure of flowlines of the attractor in $\Wi W_0$}\label{ss.struflow}
\subsubsection{Intersectional property about periodic orbits}
In this subsection, we will show  that every periodic orbit in the attractor essentially intersects with every torus fiber of the unique torus fibration structure on $W_0$ (Proposition \ref{l.openper}). The proof heavily
depends on some topological information about the  attractors: Theorem \ref{t.folext} and Lemma \ref{l.Sgood}.

\begin{lemma}\label{l.ktransitive}
$\Lambda_q$ is a transitive expanding attractor of $Y_t^q$.
\end{lemma}
\begin{proof}
Otherwise,  by spectral decomposition theorem and  the existence of Lyapunov function, there exists several transverse surfaces
$\Sigma_1, \dots, \Sigma_n$ ($n\geq2$) of $Y_t^q$ which cut $W_q$ to $m$ ($m\geq 3$) pieces $U_1, \dots, U_m$ so that
 the maximal invariant set of $Y_t^q$ on $U_i$ ($i=1,\dots, m$) is a hyperbolic basic set. We can further assume that $\partial U_m = \Sigma_n$
 and the maximal invariant set in $U_m$ is the lifted periodic orbit repeller $\beta_q$. Set $N_q$ is the path closure of $W_q -U_m$ which can  be
 thought as the corresponding lifting space of $N_0$ which supports $\Lambda_q$.
 Then $\Sigma_1 \subset N_q$ which is transverse to $Y_t^q$. By the first part of Proposition \ref{p.BrFe} and the fact that
 $N_q$  is a hyperbolic $3$-manifold, $\Sigma_1$ is a transverse torus which is isotopic to $\Sigma_n$. Then
 by the second part of  Proposition \ref{p.BrFe}, $\Sigma_1$ and $\Sigma_n$ bound a thickened torus without maximal invariant set.
 This conflicts to the fact that every $U_i$ contains a non-empty maximal invariant set.
 Therefore, $\Lambda_q$ is a transitive  expanding attractor.
\end{proof}

\begin{proposition} \label{l.openper}
Let $\gamma$ be a periodic orbit in the attractor $\Lambda$, the algebraic intersection
number of $\gamma$ and $\Sigma$ is nonzero. Here $\Sigma$ is a fiber torus of the torus fibration on
$W_0$. Equivalently, every lifting connected component $\Wi{\gamma}$  of $\gamma$ in $\Wi W_0$ is homeomorphic to $\RR$.
\end{proposition}
%We would to give two proofs. One is depending on
%the description of $C^0$ foliation on sol-manifold by HP.
%The other proof is by using Novikov Theorem and some special dynamical properties
%of $\Lambda$. We first give the first proof.

%\begin{proof}[Proof of Lemma \ref{l.openper}: way I]
%Let $l$ be the leaf in $\cF$ which contains $\gamma$. Then $l$ is an immersed annulus which is homotopic to $\gamma$.
%Then by item $3$ of Corollary \ref{c.c0fol}, the algebraic intersection
%number of $\gamma$ and $\Sigma$ is nonzero.
%\end{proof}
To prove Proposition \ref{l.openper}, we first choose a fiber torus $\Sigma_q$ in $W_q$ in a good position in the following senses.

\begin{lemma}\label{l.Sgood}
There exists a fiber torus $\Sigma_q$ transverse to $\cF_q$ such that the induced $1$-foliation $F=\cF_q \cap \Sigma_q$ satisfies
the following conditions.
\begin{enumerate}
  \item If $c$ is a compact leaf of $F$, then $c$ is contained in $\Lambda_q$.
  \item There are only finitely many compact leaves in $F$. Moreover, the holonomy of each compact leaf
  is either contracting or repelling.
  \item Let $U$ be the path closure of $W_q-\Sigma_q\cong \Sigma_q \times [0,1]$. Then there does not exist an  annulus plaque
  in $\cF_q |_W$.
  \item There also does not exist a Mobius band plaque in $\cF_q |_U$.
\end{enumerate}
\end{lemma}
\begin{proof}
By   the main theorem of \cite{BR} by Brittenham and Roberts.
 we can pick a fibered torus $\Sigma_q$ in $W_q$ so that it transversely intersects with $\cF_q$, and further assume that the intersection $1$-foliation
is $F$.
First we claim that:  $W_q-\Lambda_q$ is homeomorphic to an open solid torus with one core
 transversely intersecting to $\Sigma_q$ once.
We give a short proof of the claim here.
 Since $N_q - \Lambda_q =\bigcup_{t\geq0} Y_t^q (\partial N_q)$,
 then $N_q - \Lambda_q$ is homeomorphic to $T^2 \times [0, 1)$.
 Further notice that  $(W_q-\Lambda_q) - (W_q -N_q)=N_q - \Lambda_q$,
 therefore  $(W_q-\Lambda_q) - (W_q -N_q)$ also is homeomorphic to $T^2 \times [0, 1)$.
 By our construction, it is easy to see that
 $W_q -N_q$ is homeomorphic to an open solid torus with one core
 transversely intersecting to $\Sigma_q$ once.
 Obviously, the two points above ensure the claim.
 This claim implies the following fact: two homotopy non-vanishing loops in $W_q-\Sigma_q$ and $W_q-\Lambda_q$ never can be freely homotopic.

 Now we show item $1$. Otherwise, $l_c \subset W_q -\Lambda_q$ where $l_c$ is the leaf of $\cF_q$ containing $c$. Note that $c$ is homotopy non-vanishing
  since it is leaf of $F$. Further notice that $c$ is homotopic to a circle in
 $W_q-\Sigma_q$. Then one can quickly find a contradiction by the fact in the first paragraph. The conclusion of item $1$ is obtained.

 By item $1$, $c\subset l_c$ which is in the interior of $\Lambda_q$, the holonomy of $l_c$ (in the foliation $\cF_q$) along
 $c$ is contracting (or expanding). Therefore, the holonomy of $c$ in $F$ is contracting (or expanding). This means that
 every compact leaf of $F$ is isolated. Moreover, the set of compact leaves in $F$ is a compact subset of $\Sigma_q$ since the holonomy
 of every non-compact leaf is trivial. By these two facts, we immediately get that there are only finitely many compact leaves in $F$. Item $2$ is proved.

 By item $2$, we can  assume that $\Sigma_q$ is in the position which transversely intersects with $\cF$ with minimal number of
 compact leaves.
 Suppose there exists an annulus plaque $A_0$ in $U$.  Let $A_0\subset l_{A_0}$ where $l_{A_0}$ is a leaf of $\cF_q$.
 Then naturally there are two cases: $\partial A_0=c_1 \sqcup c_2$ in the same connected component of $\partial U$ or in the two connected components
 of $\partial U$.

 In the first case, without loss of generality, we assume that  $A_1$ bounded by $c_1$ and $c_2$ in $\Sigma_q =\Sigma_q \times \{0\}$ so that the union of
 $A_0$ and $A_1$ bounds a solid torus $V_0$ in $U$.
  By item $2$, both of $c_1$ and $c_2$ are isolated compact leaves in $F$ with attracting or repelling holonomy.
 Let $A_1'$ be a small tubular neighborhood of $A_1$ in $\Sigma_q$ so that $F$ is transverse to $\partial A_1'$.
 Then we can isotopically push $A_1'$ to $A_0'$ relative to $\partial A_1'$ so that,
 \begin{enumerate}
   \item $A_0' \subset U-V_0$ which is close to $A_0$ and transverse to $\cF_q$;
   \item there does not exist any compact leaf in $\cF_q \cap A_0'$.
 \end{enumerate}
 Then the new torus $\Sigma'=( \Sigma_q - A_1')\cup A_0'$ automatically satisfies the following conditions:
 \begin{enumerate}
   \item $\Sigma'$ is isotopic to $\Sigma_q$ in $W_0$;
   \item $\Sigma'$ is transverse to $\cF_q$;
   \item the number of compact leaves in $\cF_q \cap \Sigma'$ is less than the number of compact leaves in
   $F=\cF_q \cap \Sigma_q$.
 \end{enumerate}
 This conflicts to the assumption about the minimality of the number of the compact leaves in $F$.

  On the other hand, we suppose that there is an annulus $A$ in $V$ so that one boundary connected component is in $T^2\times
\{0\}$ and the other is in $T^2 \times \{1\}$ since $A\subset l_{\gamma}$, a leaf of $\cF_q$ and the monodromy map (from $V$ to $W_0$) is isotopic to a Thom-Anosov automorphism. Therefore, $C_{\gamma}$ is not homeomorphic to an annulus.

   Then item $3$ is obtained.

 Now we prove item $4$. Otherwise, there  is a Mobius band $B$ in $U$ so that $c=\partial B$ is an essential simple closed
 curve in $\Sigma_q$. $c$ is essential since it is a compact leaf of $F$ on $\Sigma_q$. In $H_1(U,\ZZ)$, $[c]= 2[c_B]$ where $c_B$ is a core circle of $B$.
 But notice that $H_1(U, \ZZ) \cong \ZZ \oplus \ZZ$ with $[c]$ as a generator, $[c]$ can not be equivalent to $2[c_B]$. Therefore we get a contradiction.
 Item $4$ is proved.
\end{proof}

\begin{proof}[Proof of Proposition \ref{l.openper}]
First we choose $\Sigma_q$ constructed in Lemma \ref{l.Sgood}.
Assume by contradiction that there exists a periodic orbit $\gamma\subset \Lambda$, such that the algebraic intersection
number of $\gamma$ and $\Sigma$ is $0$. Then there exists a periodic orbit by lifting, still called by $\gamma$,
which satisfies that the algebraic intersection number of $\gamma$ and $\Sigma_q$ is $0$.
Therefore, up to cyclically finite cover, we can further assume that $\gamma$ and $\Sigma_q$ are disjoint, i.e. $\gamma \subset U$.

Let $l_{\gamma}$ be the leaf in $\cF_q$ containing $\gamma$ and $C_{\gamma}$ be the connected component of $l_{\gamma}\cap U$ containing $\gamma$.
 By the fact in the first paragraph in the proof of Lemma \ref{l.Sgood}, one can easily get that $\Lambda_q \cap \Sigma_q \neq \emptyset$.
 Moreover, since $\Lambda_q$ is a transitive expanding attractor (Lemma \ref{l.ktransitive}), every separatrice of
 $l_{\gamma}^+$ is dense in $\Lambda$ where  $l_{\gamma}^+$ is a connected component of $l_{\gamma}-\gamma$.
 Therefore, $l_{\gamma}^+ \cap \Sigma\neq \emptyset$.

Now we choose two copies of $(U, \cF_q|_U)$ and glue them together along their boundaries by identity map
to get a foliation $\cH$ on $T^3$.\footnote{This technique also has been used in \cite{HP}.} Set $l_{\gamma}^H$ is the leaf in $\cH$ which contains $\gamma$.
Then $l_{\gamma}^H$ is the surface by gluing two copies of $C_{\gamma}$ by identity map on $C_{\gamma}\cap \Sigma_q$.
Note that $C_{\gamma}$ contains an essential simple closed curve $\gamma$.
Then by some elementary arguments by surface classification theorem, one can get that
$\pi(l_{\gamma}^H)$ is non-abelian except for the case that $C_{\gamma}$ is homeomorphic to either an annulus or a Mobuis band.
But item $3$ and item $4$ in Lemma \ref{l.Sgood} say that both of these two cases do not appear.
Therefore, $\pi(l_{\gamma}^H)$ is non-abelian.

Item $3$ in Lemma \ref{l.Sgood}  also implies that there is no half-Reeb component in $\cF_q|_U$, and further notice that
$\cF_q|_U$ doesn't contain any Reeb component, then $\cH$ is a non-Reeb foliation on $T^3$.
By Novikov Theorem, the non-abelian group $\pi_1(l_{\gamma}^H)$ is a sub-group of $\pi_1 (T^3)$ which is abelian. We get an contradiction.
Therefore, there does not exist a periodic orbit $\gamma$ such that the algebraic intersection
number of $\gamma$ and $\Sigma$ is $0$.
\end{proof}

\subsubsection{Long time behavior of every flowline}
  In this subsection, we will prove that every flowline of the Axiom flow essentially intersects with every torus fiber in $W_0$ in some sense
       (Proposition \ref{p.liminf}). The proof depends on Proposition \ref{l.openper} and shadowing lemma.

 First we introduce Shadowing lemma for $Y_t$ in the attractor $\Lambda$  (see \cite{PW}).

 \begin{theorem}\label{t.shadow}
 For every  $\epsilon_0 >0$ and $\delta_0 >0$, there exists $\epsilon_1>0$ such that
 for any $x_1, \dots, x_q \in \Lambda$ and $t_1, \dots, t_q >1$ such that
 $d(Y_{t_i} (x_i), x_{i+1}) \leq \epsilon_1$ (we note that $x_{k+1} =x_1$),
 then there is a point $y$ with periodic $m$,
 $$(1-\epsilon_0)\sum_{i=1}^q t_i \leq m \leq (1+\epsilon_0) \sum_{i=1}^q t_i.$$
 Moreover, there is a family of increase continuous maps
 $\{\varphi_{x_i}: [0, t_i] \to \RR^+ \}$ such that,
 $\varphi_{x_1} (0)= 0, \varphi_{x_2} (0)= \varphi_{x_1} (t_1), \dots, \varphi_{x_q} (t_q)=m$
 for any $i\in \{1,\dots, k\}$ and $t \in [0, t_q]$, $$d(Y_t (x_i), Y_{\varphi_{x_i}(t)}  (y) )\leq \delta_0.$$
 \end{theorem}

\begin{lemma} \label{l.uniformbound}
Let $\Wi x$ be a point in $\Wi{\Lambda}$ and $U$ be a compact subset of $\Wi W_0$, then there exists $R>0$ so that $\Wi Y_t (\Wi x)$ ($|t|>R$) is disjoint with $U$.
\end{lemma}
\begin{proof}
Otherwise, we can assume that $\Wi Y_t (\Wi x) \in U$ for any $t>0$ where $U$ is a compact subset of $\Wi W_0$.
By compactness of $U$, there exist $t_1, t_2 \in \RR$ such that $t_2 -t_1 >1$ and $d(\Wi Y_{t_1}(\Wi x), \Wi Y_{t_2}(\Wi x))< \epsilon_1$.
Note that the claim in Theorem \ref{t.shadow} also works for $\Wi Y_t$ and $\Wi{\Lambda}$.
Therefore, there exists a periodic orbit which is very close to $\{\Wi Y_t (\Wi x), t \in [t_1,t_2]\}$.
It conflicts to the main result in Proposition \ref{l.openper}.
\end{proof}

\begin{lemma} \label{l.cooriented}
Let $\gamma_1$, $\gamma_2$ be two periodic orbits in the attractor $\Lambda$, then the algebraic intersection
numbers  $Int(\gamma_1, \Sigma)$ and $Int(\gamma_2, \Sigma)$ satisfy that $Int(\gamma_1, \Sigma) \cdot Int(\gamma_2, \Sigma)> 0$.
\end{lemma}
\begin{proof}
First of all, notice that $\Lambda$ is an expanding attractor, then there exists some $\eta>0$ so that if the distance of two points is less than $\eta$, the two points are in a small product flow box.
We first prove the lemma in a special case that two periodic orbits $\gamma_1$ and
$\gamma_2$ satisfy that, there are two points
in $\gamma_1$ and $\gamma_2$ respectively whose distance is less than $\frac{\eta}{2}$.
We assume by contradiction that $\gamma_1$, $\gamma_2$ are two periodic orbits in the attractor $\Lambda$ so that $Int(\gamma_1, \Sigma) \cdot Int(\gamma_2, \Sigma)< 0$
in this case.
The local product property implies that there are two lifting orbits $\Wi\gamma_1$ and $\Wi\gamma_2$ of $\gamma_1$ and $\gamma_2$ respectively so that they are homolinic related. This fact induces that there are two orbits $\Wi{\alpha}$ and $\Wi{\beta}$ of $\Wi Y_t$ so that,
\begin{enumerate}
  \item $\Wi{\alpha}$ is positively asymptotic to $\Wi \gamma_1$ and negatively asymptotic to $\Wi \gamma_2$;
  \item $\Wi{\beta}$ is negatively asymptotic to $\Wi \gamma_1$ and positively asymptotic to $\Wi \gamma_2$.
\end{enumerate}
Then for $\epsilon_1 >0$, there exists a flow arc $c_{\alpha}$ in $\Wi{\alpha}$ and a flow arc $c_{\beta}$ in
$\Wi{\beta}$ so that,
\begin{enumerate}
  \item the starting point $s(c_{\alpha})$ and ending point $e(c_{\alpha})$ of $c_{\alpha}$  are in the $\epsilon_1$ neighborhoods of $\Wi \gamma_1$ and
  $\Wi\gamma_2$ respectively;
  \item $s(c_{\beta})$ and $e(c_{\beta})$ are in the $\epsilon_1$ neighborhoods of $\Wi\gamma_2$ and
  $\Wi\gamma_1$ respectively.
\end{enumerate}
Note that the Deck transformation associated to the covering map $\Wi W_0 \to W_0$ is a $\ZZ$-action generated by $\sigma$.
Since both of $\gamma_1$ and $\gamma_2$ are periodic orbits whose algebraic intersection number with $\Sigma$ are nonzero (Proposition \ref{l.openper}), there exists a nonzero integer $k$ so that both of $\Wi \gamma_1$ and
  $\Wi\gamma_2$ are invariant under the action $\sigma^k$. Therefore,  $\Wi\gamma_1$ and
  $\Wi\gamma_2$ are invariant under the action $\sigma^{sk}$ for every $s\in \ZZ$.
  Then we can choose $s\in \ZZ$ and define a new flow arc $c_{\beta}'= \sigma^{sk} c_{\beta}$ so that,
  \begin{enumerate}
    \item $s(c_{\beta}')$  and $e(c_{\beta}')$ are in the $\epsilon_1$ neighborhoods of $\Wi\gamma_2$ and
  $\Wi\gamma_1$ respectively;
    \item there exist flow arcs $c_1$ in $\Wi\gamma_1$  and $c_2$ in $\Wi\gamma_2$ so that
    $e(c_1)$, $e(c_{\alpha}')$, $e(c_2)$ and $e(c_{\beta}')$ are in the $\epsilon_1$ neighborhoods of
    $s(c_{\alpha}')$, $s(c_2)$, $s(c_{\beta}')$ and $s(c_1)$ respectively.
  \end{enumerate}
  Therefore, these four flow arcs form a $\epsilon_1$-pseudo periodic orbit.
  By Theorem \ref{t.shadow}, there exists a periodic orbit $\Wi{\gamma}$
  in $\Wi Y_t$. It conflicts to Proposition \ref{l.openper}.
Then $Int(\gamma_1, \Sigma) \cdot Int(\gamma_2, \Sigma)> 0$.

 Now by the density of the union of periodic orbits in $\Lambda$
  and the compactness of $\Lambda$, one can immediately get that  $Int(\gamma_1, \Sigma) \cdot Int(\gamma_2, \Sigma)> 0$ for every
   two periodic orbits $\gamma_1$ and $\gamma_2$ in $\Lambda$.
  \end{proof}

\begin{remark}\label{r.posint}
Without confusion, in the following of this paper, we can always assume that $Int(\gamma, \Sigma)>0$ for every periodic orbit $\gamma$ in $\Lambda$.
\end{remark}
By Lemma \ref{l.uniformbound} and Lemma \ref{l.cooriented}, we immediately  have the following consequence.

\begin{corollary}\label{c.pinfty}
Let $\Wi x \in \Wi{\Lambda}\subset \Wi W_0$, then  $\lim_{t\to \infty} P\circ\Wi Y_t (\Wi x)=\infty$.
Moreover, if $\pi(\Wi x)$ is in a periodic orbit of $\Lambda$, $\lim_{t\to -\infty} P\circ\Wi Y_t (\Wi x)=-\infty$ and $\lim_{t\to +\infty} P\circ\Wi Y_t (\Wi x)=+\infty$.
\end{corollary}

Indeed, the second consequence in Corollary \ref{c.pinfty} also is true for every point $\Wi x$ in $\Wi{\Lambda}$.

\begin{proposition}\label{p.liminf}
Let $\Wi x \in \Wi{\Lambda}\subset \Wi W_0$, then $\lim_{t\to -\infty} P\circ\Wi Y_t (\Wi x)=-\infty$ and $\lim_{t\to +\infty} P\circ\Wi Y_t (\Wi x)=+\infty$.
\end{proposition}
\begin{proof}
Assume that there exists a point $\Wi x \in \Wi{\Lambda}$ which doesn't satisfy the consequence in the proposition.
Then by Corollary \ref{c.pinfty}, $\Wi x$ should satisfy one of the following three conditions:
\begin{enumerate}
  \item $\lim_{t\to -\infty} P\circ\Wi Y_t (\Wi x)=\lim_{t\to +\infty} P\circ\Wi Y_t (\Wi x)=-\infty$;
  \item $\lim_{t\to -\infty} P\circ\Wi Y_t (\Wi x)=\lim_{t\to +\infty} P\circ\Wi Y_t (\Wi x)=+\infty$;
  \item $\lim_{t\to -\infty} P\circ\Wi Y_t (\Wi x)=+\infty$ and $\lim_{t\to +\infty} P\circ\Wi Y_t (\Wi x)=-\infty$.
\end{enumerate}
Set $x=\pi (\Wi x)$. In the first case and the third case, we suppose that $z$ is an accumulation point
of $\{Y_n (x)\}$ ($n\in \NN$) in $\Lambda$. Then
there exist $n_1, n_2 \in \NN$ so that,
\begin{enumerate}
  \item both of $Y_{n_1} (x)$ and $Y_{n_2}(x)$ are in $\frac{\epsilon_1}{2}$ neighborhood of $z$;
  \item $n_2 \gg n_1$ so that the oriented closed curve $c=a\cup b$ satisfies that $Int(c,\Sigma)<0$ where $a$ is the oriented flowline arc $Y_t (x)$ ($t\in [n_1, n_2]$)
  and $b$ is an oriented arc in $\frac{\epsilon_1}{2}$  neighborhood of $z$ starting at $Y_{n_2} (x)$ and ending in $Y_{n_1} (x)$.
\end{enumerate}
Here the fact that $\lim_{t\to +\infty} P\circ\Wi Y_t (\Wi x)=-\infty$ ensure that $Int(c,\Sigma)<0$ when $n_2 \gg n_1$.
By   Theorem \ref{t.shadow}, there is a periodic orbit $\alpha$ of $\Lambda$
which is close to $c$ in the sense that $d(Y_t (w), Y_{\varphi_{x} (t)} (y)) \leq \delta_0$ where $w=Y_{n_1}(x)$ (the starting point of $a$, $y\in \alpha$
and $\varphi_{x}: [0, t_0] \to [0,m]$ is an increase continuous map.
Here $t_0=n_2 -n_1 >0$ satisfies that $Y_{t_0} (w)=Y_{n_2}(x)$ which is the ending point of $a$.
 Therefore $\alpha$ is freely homotopic to $c$ and $Int (\alpha, \Sigma)<0$. This consequence conflicts to Remark \ref{r.posint}: $Int (\gamma, \Sigma)<0$ for every periodic orbit $\gamma$ of $\Lambda$.

Indeed, we can do a similar argument  to show that the second case also is impossible.
The only difference is that now we should choose a new point series $\{Y_{-n}(x)\}$ instead of $\{Y_n (x)\}$ ($n\in \NN$).
\end{proof}

\subsection{Global section}\label{ss.glsec}
Recall that $(W_0, Y_t)$ is an extension of $(N_0, Y_t)$ so that the maximal invariant set
of $Y_t|_V$ is an isolated periodic orbit $\beta$ which is isotopic to the core of $V$ and transverse
to the torus fibration structure $\Sigma \Wi{\times} \SS^1$.
In the remainder of this section, we assume that $Int(\beta, \Sigma)>0$. This can be done by
 choosing the orientation of $\beta$ suitably.

The purpose of this subsection is to find a global section for $(W_0,Y_t)$. The proof will heavily rely on a theorem of Fuller \cite{Fu}.
For the convenience of  readers, we introduce the theory and some related definitions below. All of them can be found in Section
$2$ of \cite{Fu}.

Let $W$ be a compact manifold and $T_t$ ($t\geq 0$) be  a nonsingular semi-flow on $W$.
We say that a continuous map $\theta: W\to \SS^1$ is an \emph{angular function} on $W$.
The net change of $\theta$ along $T_t$, $\Delta \theta (x,t): W\times [0,+\infty) \to \RR$
is the change of $\Wi \theta(x)$ to $\Wi \theta(T_t (x))$:
$\Wi \theta( T_t (x)) - \Wi \theta(x)$. Here $\Wi \theta$ is a lifting map of $\theta$
under a covering map $\rho: \RR \to \SS^1$.

 \begin{definition}\label{d.sufsec}
 An angular function $\theta$ on $W$ defines \emph{a surface of section}
 for $T_t$ if $\Delta \theta$ on each trajectory of $T_t$ is a strictly
 increasing function of $t$.
  The pre-image of $0$ under $\theta$ is called  a \emph{surface section}.
 \end{definition}

 \begin{theorem}\label{t.Fuller}\footnote{The proof of Theorem \ref{t.Fuller} is very elegant, we suggest that an interested reader
 read the proof of Theorem $1$ in [Fu]}
 If the angular function $\theta$ on $W$ has the property that
 $\Delta \theta$ is positive somewhere on each trajectory of $T_t$ ($t\geq 0$),
 i.e. for every point $x\in W$, there exists a $t_x >0$ so that $\Delta \theta (x,t_x)>0$,
 then there exists an angular function $\theta'$ in the homotopy class of $\theta$,
 which defines a surface of section for $T_t$.
 \end{theorem}

\begin{proposition}\label{p.sect}
Up to isotopy, $\Sigma$ is a global section of $Y_t$ on $W_0$, i.e. $\Sigma$ is transverse to $Y_t$ and
every point $x\in \Sigma$ will return back to $\Sigma$ in finite time along flowlines of $Y_t$.
\end{proposition}
\begin{proof}
 Define an angular function $\theta: W_0\to \SS^1$ by the projection map on a smooth torus fibration structure $\Sigma \Wi{\times} \SS^1$ on $W_0$.
Based on the following three facts:
\begin{enumerate}
  \item Proposition \ref{p.liminf};
  \item  the periodic orbit repeller $\beta$ is positively transverse to the torus fibration everywhere;
  \item  every wandering orbit of $Y_t$ is positively asymptotic to an orbit in $\Lambda$,
\end{enumerate}
one can immediately check that for  every point $x\in W$, there exists a $t_x >0$ so that $\Delta \theta (x,t_x)>0$. Then by Theorem \ref{t.Fuller}, there exists a new angular function $\theta': M\to \SS^1$  in the homotopy class of $\theta$,
 which defines a surface of section for $Y_t$.
 Moreover, this surface of section
naturally endows $W_0$ a new surface fibration structure $F'$  by pushing $\theta'^{-1} (0)$ along the trajectories of $Y_t$ so that $Y_t$ is transverse to
the fibers of $F'$ everywhere. Note that up to isotopy, surface fibration structures are unique on the sol-manifold $W_0$ (see, for instance \cite{Thu} or \cite{Fri}).  Therefore, the consequence of the proposition is followed.
\end{proof}

\subsection{Axiom A diffeomorphisms on torus} \label{ss.claAdiff}
Recall that $\Phi_A$ is an Axiom A diffeomorphism which is induced by DA-surgery on the Thom-Anosov automorphism $A$ along
the origin $O$, and the non-wandering set of $\Phi_A$ is the union of a
transitive expanding attractor $\Lambda_a^0$
and an isolated fixed point repeller $O$. The purpose of this section is to prove:

\begin{proposition}\label{p.DAdiff}
Let $\Phi:T^2 \to T^2$ be an Axiom A diffeomorphism which satisfies the following conditions.
\begin{enumerate}
  \item The non-wandering set of $\Phi$ is the union of a transitive expanding attractor $\Lambda_a$ and an isolated fixed point repeller $R$.
  \item $\Phi$ is isotopic to the Thom-Anosov automorphism $A$.
\end{enumerate}
Then $\Phi$ and $\Phi_A$ are topologically conjugate.
\end{proposition}

We first observe that there exists a semi-conjugacy map $P$ between $\Phi$ and the Thom-Anosov map $A$ by the corresponding stable-arches shrinking so that $P (R)=O$. More precisely, we have:

 \begin{lemma}\label{l.Asecj}
There exists a continuous surjective map $P: T^2 -\{R\}\to T^2$ by the corresponding
s-arches shrinking so that:
\begin{enumerate}
  \item $\Lambda_a$ has two boundary fixed points $O_1$  and $O_2$ whose free stable separatrices are
  mapped to the original fixed point $O\in T^2$ of $A$;
  \item $P\circ \Phi = A \circ P$.
\end{enumerate}
 \end{lemma}
\begin{proof}
First notice that $\Phi$ can be thought as the filling Axiom A diffeomorphism (Definition \ref{d.filA}) of $\Phi
|_{T_0}$ where $T_0$ is a filtrating neighborhood of $\Lambda_a$.
By Theorem \ref{t.sufconj}, there exists  a pseudo-Anosov homeomorphism $\Phi'$ on $T^2$ which is isotopic
to $\Phi$
and a continuous surjective map $P: T^2 - \{R\} \to T^2$ by the corresponding stable-arches shrinking
so that $P \circ \Phi = \Phi' \circ P$. Recall the assumption that $\Phi$ is isotopic to the Thom-Anosov automorphism $A$. Then $\Phi'$ and $A$ are two isotopic pseudo-Anosov homeomorphisms on $T^2$.
By Nielsen-Thurston Theorem (see for instance, \cite{FLP}), up to isotopically topological
conjugacy, we can think $\Phi'=A$. Therefore, we can think that $P\circ \Phi = A \circ P$.
Item $2$ is proved.

 Since the complement space of $\Lambda_a$ is an open disk in $T^2$, there is a unique chain-adjacent class
 of $\Phi$ on $\Lambda_a$.
 Note that every two fixed points dynamically are the same.
 Namely, let $O_1$ and $O_2$ be two fixed points of $A$, there exists
a homeomorphism $\eta:T^2\to T^2$ so that $\eta(O_1)=O_2$ and $\eta\circ A= A\circ \eta$. Then up to topological
 conjugacy, we can further assume that $P$ maps the union of the free stable separatrices to the original fixed point $O$ of $A$. As a direct consequence, there are exactly two free stable separatrices corresponding to the chain-adjacent class. Item $1$ is proved.
\end{proof}

Before the proof of this proposition, we would like to discuss a type of circles in the wandering set of
$(T^2,\Phi)$.

Consider the wandering  space $W=T^2 -\Lambda_a \cup\{R\}$ of $\Phi$, which is an open strip minus a point and homeomorphic to $S^1\times \RR$. The union of all $s$-arches of $\Lambda_a$ and the free stable separatrices of the two boundary periodic orbits
foliate $W$. We call this $1$-foliation $f^s$.
Then  the wandering orbit space of $\Phi$, $\cO$ is a torus endowed with a
$C^{1+}$ $1$-foliation $h^s$ induced by $f^s$.
Let $\pi: (W, f^s)\to (\cO, h^s)$ be the associated cyclic covering map.
By the construction of $\Phi$, we can observe that:
\begin{enumerate}
  \item there are only two compact leaves in
$h^s$ which are the image of the two free stable separatrices of the two boundary periodic orbits under $\pi$;
  \item the holonomy of every compact leaf is either contracting or repelling;
  \item the attracting directions of the two compact leaves are coherent:
  as the two boundary components of an annulus embbeded into $\cO$, they
  give the same orientation of a homotopy contracting circle of the annulus.
\end{enumerate}
In summery,
 $h^s$ induced by $f^s$ is a $C^{1+}$ $1$-foliation composed of $2$-Reeb annuli with the same asymptotic directions on the torus $\cO$. One can see  Figure \ref{f.O} as an illustration.

  \begin{figure}[htp]
\begin{center}
  \includegraphics[totalheight=7cm]{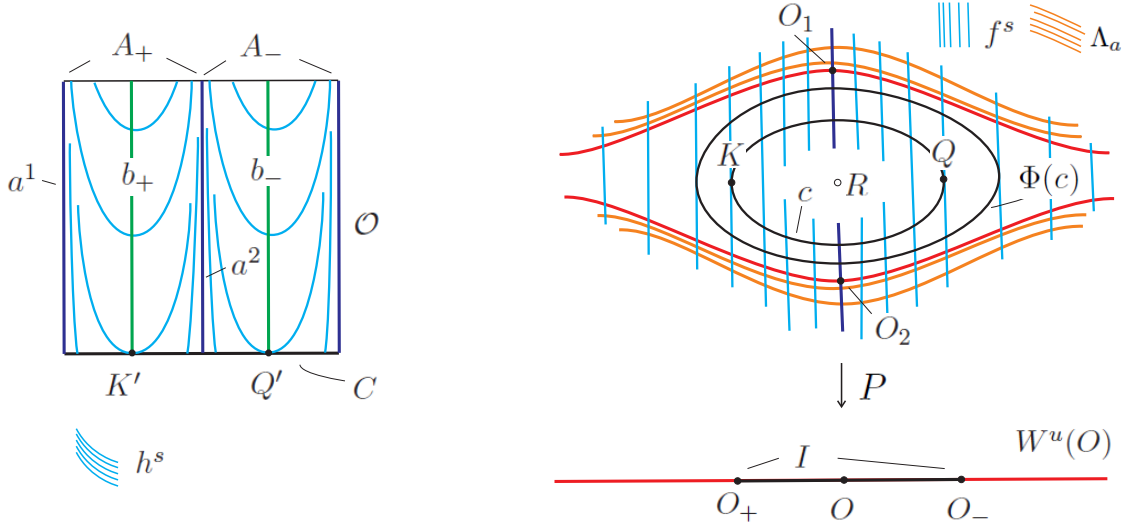}\\
  \caption{the wandering orbit space $\cO$ and a nice circle $c$}\label{f.O}
\end{center}
\end{figure}

Up to topological leaf-conjugacy, we can parameterize $h^s$ on $\cO$ as follows.
\begin{enumerate}
  \item $\cO$ is coordinated by $[-1,3]\times \RR^1$ by the action of  $\sigma$
  where $\sigma (x,y)=(x,y+1)$. Let $A_+$ and $A_-$ be the annuli parameterized by $[-1,1]\times \RR^1$
  and $[1,3]\times \RR^1$ under the $\ZZ$-action.
  \item  $h^s |_{A_+}$ and $h^s |_{A_-}$ can be parameterized by the orbits of the vector fields
  $Z_+ = \cos (\frac{1}{2}\pi x) \frac{\partial}{\partial x}+\sin (\frac{1}{2}\pi x)\frac{\partial}{\partial y}$ and $Z_- = \cos (\frac{1}{2}\pi (x-2)) \frac{\partial}{\partial x}+\sin (\frac{1}{2}\pi (x-2))\frac{\partial}{\partial y}$ under the $\ZZ$-action.
  \item Set $a^1$ and $a^2$ the compact leaves of $h^s$ parameterized by $\{-1\} \times \RR^1$ and
  $\{1\} \times \RR^1$  under the $\ZZ$-action.
  Let $b_+$ and $b_-$ be the circles parameterized by $\{0\}\times \RR^1$ and $\{2\}\times \RR^1$
  which are in the middle of $A_+$ and $A_-$ respectively.
\end{enumerate}
Let $C'$ be the circle in $\cO$ parameterized by $\{(x,0)| x\in [-1,3]\}$, then
we can further assume that $P^{-1} (C')$ is the union of infinitely many pairwise disjoint
circles in $W$. Let us give some comments about this assumption. In the case that $C'$ does not
satisfy the assumption, we can do some Dehn twist along $a^1$ or $a^2$ on $C'$ to get a new circle $C''$ in $\cO$
so that  $P^{-1} (C'')$ is the union of infinitely many pairwise disjoint
circles in $W$. If we do this Dehn twist more carefully, that is to twist in a small annulus in $A_+$
close to $a^1$ or $a^2$,\footnote{One can find more details about this careful twist in the proof of Lemma \ref{l.ncircle}.} then $C''$ also can be tangent to $h^s$ at two points $(0,0)$ and $(2,0)$.  This means that we can ensure the assumption after  some re-parameterization if necessary.

Pick two points $O_+$ and $O_-$ in the two separatrices of the unstable manifold $W^u (O)$ of $O$ for
$A$ respectively
so that $\pi\circ P^{-1}(O_+)\subset A_+$ and $\pi\circ P^{-1}(O_-)\subset A_-$. Let $I$ be the arc between $O_+$ and $O_-$ in $W^u (O)$.

\begin{definition}\label{d.ncircle}
If a circle $c\subset W$ satisfies:
\begin{enumerate}
  \item $c$ is transverse to $f^s$ except for two points $K$ and $Q$
  so that $P(K)=O_+$, $P(Q)=O_-$ and $P (c)=I$;
  \item $c$ and $\Phi(c)$ are disjoint and bound a fundamental region of $\pi$,
\end{enumerate}
then we say that $c$ is a \emph{nice circle}.
\end{definition}

Figure \ref{f.O} illustrates some notations above. The point is that $c$  bounds a filtrating neighborhood of $\Lambda_a$ so that $c$ contains $K$ and $Q$ as two tangent points.

\begin{lemma}\label{l.ncircle}
There exists a nice circle $c$ in $W$ for $\Phi$.
\end{lemma}
\begin{proof}
Recall that $P^{-1} (O_+)$ and $P^{-1} (O_-)$ are two s-arches of $\Phi$.
$l_+ = \pi(P^{-1}(O_+))$ and $l_- = \pi(P^{-1}(O_-))$ are two non-compact leaves of $h^s$
in $A_+$ and $A_-$.

By the parameterization of $h^s$ on $A_+$ and $A_-$, there is a unique point $K'$ in
$l_+ \cap b_+$ and a unique point $Q'$ in $l_- \cap b_-$.
After a re-parameterization  on $\cO$ if necessary, we can assume that
$K'=(0,0)$ and $Q'=(0,2)$. Recall that $C'$ is the circle in $\cO$ parameterized by $\{(x,0)| x\in [-1,3]\}$,
then $C'$ is transverse to $h^s$ except for the two tangent points $K'$ and $Q'$.
Since $\pi(K)=K'$ and $\pi(Q)=Q'$, there exists a circle $c'$ in $\pi^{-1}(C')$ so that
$c'$ is transverse to $f^s$ except for two points $K_m$ and $Q$
where $K_m =\Phi^m (K)$ for some integer $m$, and $K$ and $Q$ are two points in $W$ so that $P(K)=O_+$ and $P(Q)=O_-$.

Now we introduce a technique to find $c$ based on $c'$.
By the paramerization of $h^s$, one can easily check that there exist a small $\epsilon >0$
and a smooth  function $g:[0, \epsilon] \to [0,m]$  so that:
\begin{enumerate}
  \item $g(0)=0$, $g(\epsilon)=-m$ and $g'(0)=g'(\epsilon)=0$;
  \item $0<|g'(t)| \ll tan \frac{1}{2} \pi (1-\epsilon)$.
\end{enumerate}

Let $\alpha_1$ and $\alpha_2$ be the  two  arcs in $T^2$ which are parameterized by
$\{(x, g(-1-x+\epsilon))| x\in [-1,\epsilon-1]\}$  and $\{(x, g(x-1+\epsilon))| x\in [1-\epsilon, 1]\}$  under the $\ZZ$-action.
Set $C_0' = C' - \{(x, 0)| x\in [-1,\epsilon-1] \cup [1-\epsilon, 1]\}$
and $C= C_0' \cup \alpha_1 \cup \alpha_2$. See Figure \ref{f.C} as an illustration.

By the construction, $C$ and $C'$ are isotopic in $\cO$.
By the assumption of the slope about the function $g$, $C$ also is transverse to $h^s$ except for $K'$ and $Q'$.
Since $g(\epsilon)=-m$,  there exists a circle $c$ in $\pi^{-1} (C)$  so that
 $c$ is transverse to $h^s$ except for $K$ and $Q$. Therefore,
 $P(K)=O_+$, $P(Q)=O_-$ and $P (c)=I$. Moreover
since $C$ is transverse to $h^s$ except for $K'$ and $Q'$.
The facts that $C$ is a circle in $\cO$ and $\pi(c)=C$ ensure that
$c$ and $\Phi(c)$ are disjoint and bound a fundamental region of $\pi:W \to \cO$.
The properties above about $c$ ensure that $c$ is a nice circle in $W$ for $\Phi$.
\end{proof}

\begin{figure}[htp]
\begin{center}
  \includegraphics[totalheight=7cm]{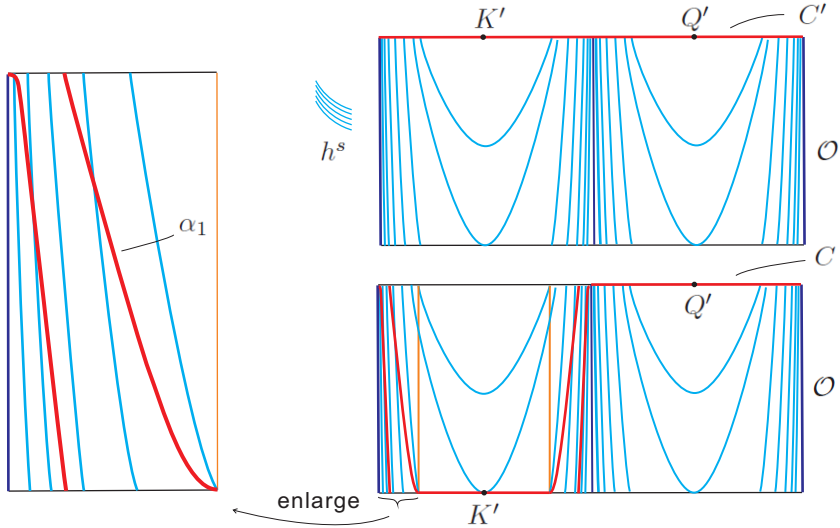}\\
  \caption{$C'$ and $C$ in the case $m=-2$}\label{f.C}
\end{center}
\end{figure}

\begin{proof}[Proof of Proposition \ref{p.DAdiff}]
On one hand, by Lemma \ref{l.Asecj}, there exists a semi-conjugacy map $P$ between $\Phi$ and the Thom-Anosov map $A$ through the corresponding stable-arches shrinking so that $P (R)=O$.
On the other hand,
automatically by the construction of DA surgery, there also exists a semi-conjugacy map $P_a$ between $\Phi_A$ and the Thom-Anosov map $A$ by the corresponding stable-arches shrinking so that $P_a (O)=O$.

%By the assumption, $D_k=T^2 - \Lambda$ is an '$k$-polygon' which contains $k$ boundary unstable leaves $l_1, \dots,
%l_k$ which contains boundary periodic points $P_1, \dots, P_k$ respectively and $k$ cusps. See Figure ? as an illustration.

Then we can automatically define a conjugacy map between $\Phi$ on $\Lambda_a$ and $\Phi_A$ on $\Lambda_a^0$ by using $P$ and $P_a$. Let us be more precise.
Let $x$ be a point in $W^u (O)- \{O\}$ where $W^u (O)$ is the unstable manifold of the origin $O$ of $A$, the set $P^{-1}(x)$ is an $s$-arch for $\Phi$ and the set $P_a^{-1}(x)$ is an $s$-arch for $\Phi_A$. Then, we can naturally choose a homeomorphism $\chi$ between $\Lambda_a$ and $\Lambda_a^0$ so that
 $\chi$ sends the two ends of $P^{-1}(x)$ to the two ends of $P_a^{-1}(x)$. One can further easily check that for every point $y\in \Lambda_a$,
$z=\chi(y) \in \Lambda_a^0$ satisfies that $P(y)=P_a (z)$ and $\chi$ is a conjugacy map between $\Phi \mid_{\Lambda_a}$ and
$\Phi_A \mid_{\Lambda_a^0}$.
Now we are left to extend $\chi$ to a conjugacy map    $h$ between $\Phi$ and $\Phi_A$.

 By Lemma \ref{l.ncircle},
there exists a nice circle $c$ in $W$ for $\Phi$ which is tangent to $f^s$ at $K$ and $Q$
so that $P(K) =O_+$, $P(Q) =O_-$  and $P(c)=I$.
We can similarly define $W_a$, $\cO_a$, $\pi_a: W_a \to \cO_a$, $f_a^s$ and $h_a^s$ for the DA Axiom A diffeomorphism $\Phi_A$. Note that Lemma \ref{l.ncircle} also works for $\Phi_A$.
Therefore, there exists a nice circle $c_a$ in $W_a$ for $\Phi_A$ which is tangent to $f_a^s$ at $K_a$ and $Q_a$
so that  $P_a(K_a) =O_+$, $P_a(Q_a) =O_-$ and $P_a(c_a)=I$.

  By the properties which $c$ and $c_a$ satisfy, we can easily build a homeomorphism $\tau: c\to c_a$ so that for every point $y\in c$, $P(y)= P_a \circ \tau(y)$,
  and in particular, $\tau(K)= K_a$ and $\tau(Q) =Q_a$. $\tau$ routinely induces a homeomorphism $t: C\to C_a$ so that
  for every point $y\in c$, $\pi_a \circ \tau (y)= t \circ \pi (y)$.
 The two commutative equalities  $P= P_a \circ \tau$ and $\pi_a \circ \tau = t \circ \pi$ induce that $P \circ \pi = P_a \circ \pi_a \circ t$. This commutative equality ensures that
  $t:C\to C_a$ can be extended to a homeomorphism
  $g:\cO \to \cO_a$ so that $g$ maps every leaf of $h^s$ to a leaf of $h_a^s$.
  By the cyclic covering maps $\pi$ and $\pi_a$, $g$ can be uniquely lifted to a homeomorphism $G:W\to W_a$
  so that $\pi_a \circ G=g\circ \pi$ and $G\mid_c= \tau$ which maps $c$ to $c_a$.

  We claim that $G$ satisfies the following conditions.
  \begin{enumerate}
    \item $G$ maps every leaf of $f^s$ to a leaf of $f_a^s$. In particular, $G$ maps the  two free stable separatrices of the two boundary periodic orbits of $\Lambda_a$ to the two free stable separatrices of the two boundary periodic orbits of $\Lambda_a^0$.
    \item $\pi_a \circ G=g\circ \pi$.
    \item $G\circ \Phi= \Phi_A \circ G$.
    \item Set $c$ and $c_a$ bound disks $D$ and $D_a$ in $T^2$ respectively, then $\Phi^{-1} (D)$ and $\Phi_A^{-1} (D_a)$ are in the interior of $D$ and $D_a$ respectively, and moreover
       $\bigcap_{k=0}^{+\infty} \Phi^{-k} (D) =R$ and  $\bigcap_{k=0}^{+\infty} \Phi_A^{-k} (D_a) =O$.
    \end{enumerate}

    The proof of the claim  are followed by some routine checks. We just list the main ideas here.
    Item $1$ is followed by the facts that $g$ maps every leaf of $h^s$ to a leaf of $h_a^s$, and in particular, $g$ maps the two compact leaves of $h^s$ to the two compact leaves of $h_a^s$.
    Item $2$ is one condition which $G$ already satisfies.
    Item $3$ is followed by item $2$ and the fact that  $G$ is the lifting map of $g$
    so that $G\mid_{c}= \tau$ which maps $c$ to $c_a$.
    To show item $4$,  we only need to show that $\Phi^{-1} (D)$ is in the interior of $D$.
    This can be followed by the fact that $D- \Phi^{-1} (D)$ is the fundamental region of the cyclic covering map $\pi:W\to \cO$  so that
    $\pi(\partial D) = \pi(\partial (\Phi^{-1} (D)))= C$ and therefore $D- \Phi^{-1} (D)$ is homeomorphic to
    $\SS^1 \times [0,1)$.

    By item $1$, item $2$ and item $4$ above, $G$ can be uniquely extended to a homeomorphism $h: T^2 \to T^2$ so that
    $h\mid_W =G$, $h\mid_{\Lambda_a}=\chi$ and $h(R)=O$.
    Furthermore, by item $3$ above and the two facts that $h(R)=O$ and $h\mid_{\Lambda_a}= \chi$,
    one can immediately get that $h$ is a conjugacy map
    between $\Phi$ and $\Phi_A$, namely, $h\circ \Phi= \Phi_A \circ h$.
\end{proof}

\subsection{End of the proof of Theorem \ref{t.claexp}} \label{ss.fclaexp}

\begin{proof}[Proof of Theorem \ref{t.claexp}]
By Proposition \ref{p.sect}, there exists a global torus section $\Sigma$ of $(W_0, Y_t)$. Let $\Phi$
be  the self-homeomorphism on $\Sigma$ induced by the first return map  of $Y_t$. Then $W_0$ is homeomorphic to
 the mapping torus of $\Phi$. Note that as the sol-manifold induced by the Anosov automorphism $A$, $W_0$ has a unique torus fibration structure (see, for instance \cite{Thu} or \cite{Fri}).
 Therefore, $\Phi$ is isotopic
to either $A$ or $A^{-1}$.

When $A=\left(\begin{array}{cc}
                                         2 & 1 \\
                                         1 & 1 \\
                                       \end{array}
                                     \right)$,
                                     $A$ and $A^{-1}=\left(\begin{array}{cc}
                                         1 & -1 \\
                                         -1 & 2 \\
                                       \end{array}
                                     \right)$ are conjugate by
                                     $B=\left(
                                          \begin{array}{cc}
                                            0 & 1 \\
                                            -1 & 0 \\
                                          \end{array}
                                        \right)$.
    This fact promises us to only  focus on the case that $\Phi$ is isotopic to $A$.
    By Proposition \ref{p.DAdiff}, $\Phi$ and $\Phi_A$ are topologically conjugate.
                                        As a direct consequence, $(W_0, Y_t)$ and $(W_0, Y_t^0)$ are topologically equivalent. Here $Y_t^0$ is the suspension Axiom $A$ flow induced by $\Phi_A$.
                                      Also recall the fact that for $3$-dimensional flows case, up to topological equivalence, there is a unique filtrating neighborhood for a given expanding attractor.
                                      Then, we can conclude that  $(N_0, Y_t)$ and $(N_0, Y_t^0)$ are topologically equivalent.
\end{proof}

\section{non-transitive Anosov flows}\label{s.nonAno}

Let $(N_1, Y_t^1, \Lambda_1, T_1)$ and $(N_2, Y_t^2, \Lambda_2, T_2)$
be two copies of $(N_0, Y_t^0, \Lambda_0, \partial N_0)$.
$Y_{-t}^2$ is a nonsingular flow on $N_2$  so that
the maximal invariant set of $Y_{-t}^2$ is an expanding repeller $\Lambda_2$
supported by $N_2$.
We still denote by $F^s$ the intersectional $1$-foliation of $T_1$ and the stable foliation $\cF^s$ of $\Lambda_1$.
In the system $(N_2, Y_{-t}^2)$, we denote by $F^u$ the intersectional $1$-foliation of $T_2$ and the unstable foliation $\cF^u$ of $\Lambda_2$. Here $(F^u, \cF^u)$ exactly corresponds to $(F^s, \cF^s)$ of $(N_2, Y_t^2)$ when we coordinate $(N_2, Y_t^2)$ by $(N_0, Y_t^0)$.
Franks and Williams constructed their  non-transitive Anosov flow $Z_t^0$  on $M_0$ which is topologically
equivalent to $(M_1, Z_t^1)$ constructed as follows.
Choose  a suitable gluing homeomorphism
$\Psi_1: T_2 \to  T_1$ so that $\Psi_1$ is isotopic to the identity map \footnote{This identity map can be understood as the identity map on $\partial N$, which coordinates both of $T_1$ and $T_2$. In the remainder of this section, we will always define an identity map in the same manner as above.} and $\Psi_1 (F^u)$ and $F^s$ are transverse everywhere. The glued flow, the glued manifold and the glued torus are named by $Z_t^1$, $M_1$ and $S_1$
respectively.

The purpose of this section is to prove Theorem \ref{t.claAno}, i.e. up to topological equivalence, every non-transitive Anosov flow $X_t$ on $M_0$ is topologically equivalent to $Z_t^0$.
Now we only need  to prove that $X_t$ on $M_0$ is topologically equivalent to $Z_t^1$ on $M_1$.
To prove the theorem, the first  step is to show that every non-transitive Anosov flow $X_t$ on $M_0$
admits a  similar decomposition structure with  $Z_t^1$:

\begin{lemma}\label{l.decomp}
Let $X_t$ be  a non-transitive Anosov flow on $M_0$. Then $X_t$ is topologically equivalent to the glued flow $Z_t^2$ on the glued manifold $M_2$ obtained by gluing $(N_1, Y_t^1)$ and $(N_2, Y_{-t}^2)$ together through a homeomorphism
 $\Psi_2: T_2 \to T_1$ so that $\Psi_2(F^u)$ is isotopic to the identity map and  transverse to $F^s$ everywhere. We denote by $S_2$ the glued torus.
\end{lemma}

\begin{proof}
Since $X_t$ is a non-transitive Anosov flow, there are at least two basic sets in the nonwandering set of $X_t$.
Let $g:M_0\to \mathbb{R}$ be a Lyapunov function of $X_t$. Then $g$ contains at least one maximum value and one minimum
value. Let $y$ be a regular value and $T$ be a connected component of $g^{-1}(y)$. Then $T$ is a
transverse surface of $X_t$. By the first conclusion of Proposition \ref{p.BrFe} due to Brunella \cite{Br} and Fenley \cite{Fen2}, $T$ is a transverse torus of $X_t$.  Since the interior of $N_0$ is homeomorphic to figure eight knot complement space, which is a hyperbolic $3$-manifold. Then, up to isotopy, $T$ is the unique incompressible torus in $M_0$ which cuts $M_0$ to two parts so that the path closure of each of them is homeomorphic to $N_0$.
Moreover,
by the second conclusion of Proposition \ref{p.BrFe}, we immediately have:
\begin{enumerate}
  \item up to isotopy along flowlines, $T$ is the unique regular level set of $g$;
  \item the nonwandering set of $X_t$ exactly is the union of an expanding attractor and an expanding repeller with $N_1'$ and $N_2'$ as their filtrating neighborhoods respectively, where the union of $N_1'$  and $N_2'$ is the path closure of $M-T$  and each of $N_1'$ and $N_2'$ is homeomorphic to $N_0$.
\end{enumerate}

 By Theorem \ref{t.claexp}, either $(N_1', X_t)$ or $(N_2', X_{-t})$ is  topologically equivalent to $(N_0,Y_t^0)$. For simplicity, up to topological equivalence, we can assume that $(N_1', X_t) =(N_1,Y_t^1)$  and $(N_2',X_t)= (N_2,Y_{-t}^2)$. Notice the facts that  both of the stable foliation and the unstable foliation of
 $X_t$ are transverse to $T$ and the intersectional $1$-foliation of these two foliations exactly is the $1$-foliation induced by the flowlines of $X_t$. Then we can assume that
 $X_t$ is topologically equivalent to the flow obtained by gluing $(N_1, Y_t)$ and $(N_2, Y_{-t})$ through a homeomorphism
 $\Psi_2: T_2 \to T_1$  so that $\Psi_2(F^u)$ is isotopic to the identity map and transverse to $F^s$ everywhere. Lemma \ref{l.decomp} is proved.
 \end{proof}

By Lemma \ref{l.decomp},  to prove Theorem \ref{t.claAno}, we only need to show the following claim.
\begin{claim}\label{cl.final}
$Z_t^1$ on $M_1$ and $Z_t^2$ on $M_2$ are
topologically equivalent.
\end{claim}
From now on, we focus on proving this claim. The proof will be divided into the following three steps which are corresponding to Section \ref{ss.lbh}, Section \ref{ss.g1H1} and Section \ref{ss.ep1} respectively.
\begin{enumerate}
  \item [Step 1.] First we lift every object in question to a infinitely many cyclic cover of $ N_i$ ($i\in \{1,2\}$), say $\Wi N_i$.
  Then we build a self-homeomorphism $\Wi h_1$ on  $\Wi T_1$ which  preserves the lift of $F^s$
  and  preserves the lifts of $\Psi_2(F^u)$ and $\Psi_1(F^u)$ with some  commutative conditions (Lemma \ref{l.ch1}). Here $\Wi T_1$  is the lift of $T_1$.
  \item [Step 2.] We construct a height function $\Wi g_1: \Wi T_1 \to \RR$, and then with the help of $\Wi g_1$,
  we can extend $\Wi h_1 :\Wi T_1 \to \Wi T_1$ to an orbit-preserving homeomorphism map $\Wi H_1$
  on the lifting wandering set  so that, when close to the lifting attractor, $\Wi H_1$ totally preserving strong
  s-arches. Then we can show that $\Wi H_1$ can be extended to an orbit-preserving homeomorphism on $\Wi N_1$.
  We still call by $\Wi H_1$ the extended homeomorphism.
  \item [Step 3.] $\Wi H_1$  induces an orbit-preserving homeomorphism $H_1$ on $(N_1, Y_t^1)$ which is
  an extension of the homeomorphism $h_1: T_1 \to T_1$ induced by $\Wi h_1$.
  Similarly, we can build an orbit-preserving homeomorphism $H_2$ on $(N_2, Y_t^2)$ which is
  an extension of a homeomorphism $h_2: T_2 \to T_2$ which is close related to $h_1$ (Corollary \ref{c.ch2} and Lemma \ref{l.H2}).
  The relationships between $h_1$ and $h_2$ can ensure us to pinch $H_1$ and $H_2$ to an orbit-preserving homeomorphism $H: (M_2, Z_t^2) \to (M_1, Z_t^1)$.
\end{enumerate}

\subsection{Lifting and boundary homeomorphism $\Wi h_1$}\label{ss.lbh}
Recall that $(M_i, Z_t^i)$
 ($i=1,2$)
 is topologically equivalent to the flow obtained by gluing $(N_1, Y_t)$ and $(N_2, Y_{-t})$ through a homeomorphism
 $\Psi_i: T_2 \to T_1$ so that $\Psi_i$ is isotopic to the identity map and $\Psi_i(F^u)$ is transverse to $F^s$ everywhere.

We will work on a type of cyclic covering spaces of $N_1$ and $N_2$, so we have to introduce
more information  about the covering spaces.
We can use some notations introduced to $(N_0,Y_t^0)$ in Section \ref{sss.epaflow} for both of $(N_1, Y_t^1)$  and $(N_2, Y_{-t}^2)$.
Let $\Wi N_i$ ($i\in\{1,2\}$) be the $\ZZ$-cyclic cover of $N_i$ associated to $\Sigma_0$ and $\pi_i: \Wi N_i \to N_i$ be  the corresponding covering map.
Let $P: \RR \to \SS^1$ be a covering map. Then one can easily use the standard  lifting criterion to lift
$P_i: N_i \to \SS^1$ to $\Wi P_i: \Wi N_i \to \RR$ so that $P\circ \Wi P_i = P_i \circ \pi_i$.
Naturally, $\{ \Wi{\Sigma}_t = \Wi P_i^{-1} (t) | t \in \RR\} $ induces
 a product fibration structure on $\Wi {N_i}$ so that $\pi_i (\Wi{\Sigma}_t ) = \Sigma_t$.
Further recall that $P_i\circ Y_t (x)= t+P_i (x) \in \SS^1$ for every $x\in T_i$,
it is easy to check that, $\Wi P_i (\Wi Y_t (\Wi x))= \Wi P_i (\Wi x) +t$.
The corresponding deck-transformation of the covering map $\pi_i:\Wi {N_i} \to N_i$ is a $\ZZ$-action which is generated by $\sigma_i$ where $\sigma_i$ satisfies that for every $\Wi x \in \Wi \Sigma_t$, $\sigma_i (\Wi x)\in \Wi \Sigma_{t+1}$.
Certainly we have $\pi_i \circ \sigma_i = \pi_i$ and $\Wi P_i \circ \sigma_i = \Wi P_i +1$.

Let $\Wi T_i, \Wi Y_t^i$ and $\Wi \Lambda_i $ be the corresponding lifts of
$T_i, Y_t^i$ and $\Lambda_i$. Here, the existence of $\Wi Y_t^i$ is followed by
the standard homotopy lifting property. Automatically, $\Wi Y_t^i$ is commutative to the corresponding deck transformation, i.e. $\Wi Y_t^i \circ \sigma_i =\sigma_i \circ \Wi Y_t^i$.
Naturally, $\pi_i: \Wi T_i \to T_i$ also is a cyclic covering map.
We further assume that the foliations
$\Wi{\cF}^s$, $\Wi{\cF}^u$, $\Wi F^s$ and $\Wi F^u$ on $\Wi N_1$, $\Wi N_2$, $\Wi T_1$ and $\Wi T_2$
are the lifts of $\cF^s$, $\cF^u$, $F^s$ and $F^u$ on $N_1$, $N_2$, $T_1$ and $T_2$ respectively.

 By covering lift Theorem once more,  there exists a lifting homeomorphism $\Wi\Psi_i: \Wi T_2 \to \Wi T_1$ ($i\in\{1,2\}$) of $\Psi_i: T_2 \to T_1$ commutative to the corresponding deck transformations:
 $\Wi \Psi_i \circ \sigma_2 =\sigma_1 \circ \Wi\Psi_i$.

We collect all of the commutative conditions above in the following lemma which will be very useful in the following of this section.

\begin{lemma}\label{l.comm}
$\pi_i$, $\sigma_i$, $P$, $\Wi P_i$, $P_i$ and $\Phi_i$ share the following commutative equalities:
\begin{enumerate}
\item $\pi_i \circ \sigma_i= \pi_i$;
\item $P\circ \Wi P_i = P_i \circ \pi_i$;
  \item $\Wi Y_t^i \circ \sigma_i =\sigma_i \circ \Wi Y_t^i$;
    \item $\Wi P_i \circ \Wi Y_t^i (\Wi x)= \Wi P_i (\Wi x) +t$ and $\Wi P_i \circ \sigma_i = \Wi P_i +1$;
    \item $\pi_1 \circ \Wi \Psi_i = \Psi_i \circ \pi_2$ and $\sigma_1 \circ \Wi{\Psi_i}=\Wi{\Psi_i} \circ \sigma_2$.
\end{enumerate}
\end{lemma}

Now we start to understand the two foliations $F^s$ and $\Psi_i (F^u)$ on $S_i$ which are transverse everywhere.
Recall that $F^s$ and $F^u$ are two $1$-foliations on $T_1$ and $T_2$ respectively so that everyone is the union of two Reeb annuli which are asymptotic to the same direction along every compact leaf.
So we would like to abstractly understand two foliations $F^1$ and $F^2$ on a torus $S$ so that
$F^i$ ($i\in\{1,2\}$) is the union of such type of two Reeb annuli and they are transverse everywhere.

We can build two transverse $1$-foliations $F_0^1$ and $F_0^2$ on a torus $T^2$  which satisfy the conditions above as follows (Figure \ref{f.2twoReeb}).
First we build two $1$-foliations $\Wi F_0^1$ and $\Wi F_0^2$ on $\RR^2$
induced by the orbits of the vector fields   $Z^1 = \cos \pi x \frac{\partial}{\partial x}+\sin \pi x\frac{\partial}{\partial y}$ and $Z^2 = \sin \pi x\frac{\partial}{\partial x}-\cos \pi x\frac{\partial}{\partial y}$ respectively.
Note that $\Wi F_0^1$ and $\Wi F_0^2$ are transverse everywhere and they are invariant under the
$\ZZ \oplus \ZZ$ action on $\RR^2$ by $(m,n)\circ (x,y)=(m+x, n+y)$ for every $(m,n)\in \ZZ\oplus \ZZ$. Then $(\Wi F_0^1,\Wi F_0^2)$ induces a $1$-foliation pair $(F_0^1, F_0^2)$ on $T^2$ so that
they are transverse everywhere. Moreover, one can automatically check that each of  $F_0^1$ and  $F_0^2$ is the union of the desired two Reeb annuli.

\begin{figure}[htp]
\begin{center}
  \includegraphics[totalheight=5.5cm]{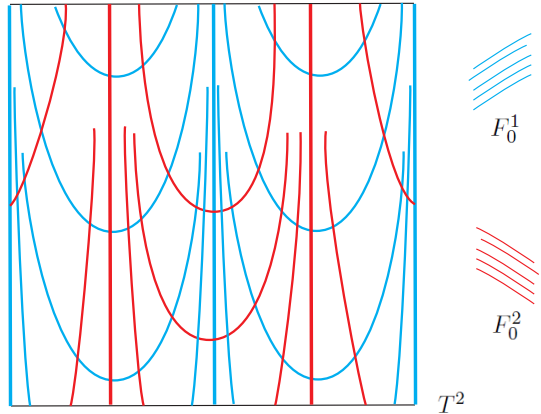}\\
  \caption{$F_0^1$ and $F_0^2$ on $T^2$}\label{f.2twoReeb}
\end{center}
\end{figure}

By some standard and elementary technique (for instance, the similar techniques can be found in Section $7$ of \cite{BBY}), one can easily show the following proposition:

\begin{proposition}\label{p.stbifol}
Let $F^1$ and $F^2$ be two foliations on a torus $S$ so that
each of $F^1$ and $F^2$ is the union of two Reeb annuli which are asymptotic to the same direction along every compact leaf and they are transverse everywhere.
Then there exists a homeomorphism $\varphi: S\to T^2$ so that $\varphi$ preserves leaves between $F^i$
and $F_0^i$.
\end{proposition}

\begin{remark}\label{r.samecmpt}
Note that if we pick a noncompact leaf $l^2$ in $F_0^2$, $l^2$ intersects with one of the two compact leaves of $F_0^1$ once and is disjoint to the other.
By the comments above and Proposition \ref{p.stbifol}, it maybe appear the following phenomena:
there is a noncompact leaf $l^u$ in $F^u$ so that $\Psi_1 (l^u)$ is disjoint to one
of the two compact leaves in $F^s$ and $\Psi_2 (l^u)$ is disjoint to the other.
This phenomena will provide some troubles in the following of the paper. We would like to give some comments to overcome it at the moment. Note that there is self-conjugacy map  $E=\left(
                                          \begin{array}{cc}
                                            -1 & 0 \\
                                            0 & -1 \\
                                          \end{array}
                                        \right)$
for the Anosov automorphism $A$ so that $E$ exchanges the two stable sepratrices of $O$.
The existence of $E$ implies that there is a self orbit-preserving map $H_E$ of $(N_0, Y_t^0)$
satisfying that $H_E$ exchanges the two boundary periodic orbits of $Y_t^0$.
This fact permits us to assume that  $\Psi_1 (l^u)$ and  $\Psi_2 (l^u)$
are disjoint to the same compact leaf of $F^s$ by a further self orbit-preserving map $H_E$
if necessary. We keep this assumption from now on.
\end{remark}

%We assume that $m_1^s$ and  $m_2^s$ are two compact leaves of $F^s$ which cut $T_1$ to two Reeb annuli $A_+^s$ and $A_-^s$. Similarly we assume that $m_1^u$ and  $m_2^u$ are two compact leaves of $F^u$ which cut $T_2$ to two Reeb annuli $A_+^u$ and $A_-^u$.

%\begin{claim}\label{c.01}
%\begin{enumerate}
 % \item $\Psi(m_1^u)$ is in the interior of one Reeb annulus of $F^s$ and $\Psi(m_2^u)$ is in the interior of the other Reeb annulus;
%  \item $\Psi(m_j^u)$  intersects with every leaf  of $F^s$ at most once and symmetrically, $m_i^s$ intersects with every leaf  of $\Psi(F^u)$ at most once ($i,j \in \{1,2\}$).
%\end{enumerate}
%\end{claim}
%\begin{proof}
%To show item $1$, we only need to show the following two points:
%\begin{enumerate}
%  \item $\Psi(m_i^u)$ is disjoint to $m_j^s$ ($i,j\in \{1,2\}$);
 % \item $\Psi(m_1^u)$ and $\Psi(m_2^u)$ are not belong to the same Reeb annulus.
%\end{enumerate}

%TO BE FILLED...
%\end{proof}

\begin{lemma}\label{l.ch1}
There exists a unique homeomorphism $\Wi h_1: \Wi T_1  \to \Wi T_1$ so that:
\begin{enumerate}
  \item $\Wi h_1  (\Wi l^s) = \Wi l^s $ for every leaf $\Wi l^s$ of $\Wi F^s$;
  \item $\Wi h_1  \circ \Wi \Psi_1 (\Wi l^u) = \Wi \Psi_2  (\Wi l^u)$ for every leaf $\Wi l^u$ of $\Wi F^u$;
  \item $\Wi h_1 \circ \sigma_1 = \sigma_1 \circ \Wi{h_1}$.
\end{enumerate}
\end{lemma}
\begin{proof}
First we remind the reader that  we will always use the model $(F_0^1, F_0^2)$ on $T^2$ to understand
$(F^s, \Psi_i (F^u))$ on $S_i$ by Proposition \ref{p.stbifol} and
we will omit the routine checks about some intuitive properties implied by the model.
Pick a point $\Wi x \in \Wi T_1$, by Proposition \ref{p.stbifol}, there exists a unique leaf
$\Wi l^s$ of $\Wi F^s$ and  a  unique leaf $\Wi l^u$ of $\Wi F^u$ so that
$\{ \Wi x\} = \Wi l^s \cap \Wi \Psi_1 (\Wi l^u)$.
By the assumption in Remark \ref{r.samecmpt} and Proposition \ref{p.stbifol} once more,
one can find a unique point in $\Wi l^s \cap \Wi \Psi_2 (\Wi l^u)$, defined by $\Wi h_1  (\Wi x)$.

One can build $\Wi \tau_1: \Wi T_1 \to \Wi T_1$ in a similar way so that
$\Wi h_1\circ \Wi \tau_1= \Wi \tau_1\circ \Wi h_1 =Id$, which is the identity map on $\Wi T_1$.
Therefore, $\Wi h_1$ is a bijection map. By the construction of $\Wi h_1$,
$\Wi h_1$ preserves the local product structures induced by $(\Wi F^s, \Psi_1 (\Wi F^u))$
and $(\Wi F^s, \Psi_2 (\Wi F^u))$. Therefore, both of $\Wi h_1$ and $\Wi \tau_1$ are continuous.
Hence, $\Wi h_1$ is a homeomorphism. Item $1$ and item $2$ of the lemma  and the uniqueness of
$\Wi h_1$
can be directly followed by the construction of $\Wi h_1$. We are left to prove item $3$.

Pick a point $\Wi x \in \Wi T_1$ so that $\{ \Wi x\} = \Wi l^s \cap \Wi \Psi_1 (\Wi l^u)$.
We have£¬
\begin{eqnarray*}
   \sigma_1 \circ \Wi{h_1}(\Wi x) &=& \sigma_1 ( \Wi l^s \cap \Wi \Psi_2 (\Wi l^u))\nonumber\\
                           &=& \sigma_1 ( \Wi l^s) \cap (\sigma_1 \circ \Wi \Psi_2 (\Wi l^u))\nonumber\\
                           &=&\sigma_1 ( \Wi l^s) \cap  (\Wi \Psi_2 \circ  \sigma_2 (\Wi l^u)),
\end{eqnarray*} and,
\begin{eqnarray*}
   \sigma_1 (\Wi x) &=& \sigma_1 (\Wi l^s \cap \Wi \Psi_1 (\Wi l^u))\nonumber\\
                           &=& \sigma_1 (\Wi l^s) \cap (\sigma_1 \circ \Wi \Psi_1 (\Wi l^u))\nonumber\\
                           &=&\sigma_1 (\Wi l^s) \cap (\Wi \Psi_1 \circ \sigma_2 (\Wi l^u)).
\end{eqnarray*}
The third equality in each recursive equation above is followed by item $5$ of Lemma \ref{l.comm}.
Therefore, by the construction of $\Wi h_1$, $\Wi h_1 \circ \sigma_1 (\Wi x) =
\sigma_1 (\Wi l^s) \cap (\Wi \Psi_2 \circ \sigma_2 (\Wi l^u))$.
In summary, $\Wi h_1 \circ \sigma_1 (\Wi x) = \sigma_1 \circ \Wi{h_1} (\Wi x)$.
\end{proof}

The following corollary is a direct consequence of Lemma \ref{l.ch1} and Lemma \ref{l.comm}.
\begin{corollary}\label{c.ch2}
$\Wi h_1$ induces a homeomorphism $\Wi h_2= \Wi \Psi_2^{-1}\circ \Wi h_1\circ \Wi \Psi_1: \Wi T_2 \to \Wi T_2$ so that:
\begin{enumerate}
  \item $\Wi h_2 \circ \Wi \Psi_1^{-1}(\Wi l^s) = \Wi \Psi_2^{-1}(\Wi l^s)$ for every leaf $\Wi l^s$ in $\Wi F^s$;
 \item $\Wi h_2 (\Wi l^u) = \Wi l^u$ for every leaf $\Wi l^u$ in $\Wi F^u$;
  \item $\Wi h_2 \circ \sigma_2 = \sigma_2 \circ \Wi h_2$.
\end{enumerate}
\end{corollary}

\subsection{Hight function $\Wi g_1$ and topologically equivalent map $\Wi H_1$}\label{ss.g1H1}

Define a continuous function $\Wi g_1: \Wi T_1 \to \RR$ by $\Wi g_1 (\Wi x)= \Wi P_1 (\Wi x)-\Wi P_1 \circ \Wi h_1 (\Wi x)$ for every point $\Wi x \in \Wi T_1$.

\begin{proposition}\label{p.cg1}
$\Wi g_1$ satisfies the following properties:
\begin{enumerate}
  \item $\Wi g_1\circ \sigma_1 = \Wi g_1$;
  \item there exists a big integer $K> 0$ so that $|\Wi g_1 (\Wi x)|< K$ for every point $\Wi x \in \Wi T_1$;
  \item $\Wi P_1 \circ \Wi Y_K^1(\Wi x)= \Wi P_1 \circ \Wi Y_{K+\Wi g_1 (\Wi x)}^1\circ\Wi h_1(\Wi x)$.
\end{enumerate}
\end{proposition}
\begin{proof}
Let $\Wi x$ be a point in $\Wi T_1$. Then we have,
\begin{eqnarray*}
   \Wi g_1\circ \sigma_1  (\Wi x) &=& \Wi P_1 \circ \Wi h_1 \circ \sigma_1 (\Wi x) - \Wi P_1 \circ \sigma_1 (\Wi x)\nonumber\\
                           &=& \Wi P_1 \circ \sigma_1\circ \Wi h_1 (\Wi x)-  \Wi P_1 \circ \sigma_1 (\Wi x)\nonumber\\
                           &=&(\Wi P_1 \circ \Wi h_1 (\Wi x)+1) -(\Wi P_1 (\Wi x) +1)\nonumber\\
                           &=& \Wi P_1 \circ \Wi h_1 (\Wi x) -\Wi P_1 (\Wi x) \nonumber\\
                           &=& \Wi g_1 (x).
\end{eqnarray*}
Here the second equality is because of item $3$ of Lemma
\ref{l.ch1} and the third equality is because of item $4$ of Lemma \ref{l.comm}.  Then item $1$ of the proposition is proved.

Item $2$  is a direct consequence of item $1$ and the fact that $\Wi g_1$ is continuous.
Let us explain more. Item $1$ ensure that the continuous map $\Wi g_1$ can naturally induce a continuous map $g_1: T_1 \to \RR$ so
that $g_1 \circ \pi_1 = \Wi g_1$.  Since $g_1$ is a continuous function on a compact set $T_1$, the image of $g_1$ is uniformly bounded. Moreover, it is obvious that the images of $g_1$ and $\Wi g_1$ are the same. Therefore, the image of $\Wi g_1$ is uniformly  bounded.

Now we check item $3$. By item $4$ of Lemma \ref{l.comm},  $\Wi P_1 \circ \Wi Y_{K+\Wi g_1 (\Wi x)}^1\circ\Wi h_1(\Wi x)=
 \Wi P_1 \circ \Wi h_1(\Wi x) +K+g_1 (\Wi x)$ and $\Wi P_1 \circ \Wi Y_K^1(\Wi x)=
 \Wi P_1 (\Wi x)+K$. Moreover, recall that $\Wi g_1 (\Wi x)= \Wi P_1 (\Wi x)-\Wi P_1 \circ \Wi h_1 (\Wi x)$, one can immediately get that $\Wi P_1 \circ \Wi Y_{K+\Wi g_1 (\Wi x)}^1\circ\Wi h_1(\Wi x)=
 \Wi P_1 \circ \Wi Y_K^1(\Wi x)$.
\end{proof}

Now we define a map $\Wi H_1 : \Wi N_1 \to \Wi N_1$ in  different ways depending on whether
the pre-image belongs to $\Wi \Lambda_1$ or not.
If $\Wi x \in \Wi \Lambda_1$, $\Wi H_1 (\Wi x)$ is defined to be $\Wi x$. In another word,
 $\Wi H_1 |_{\Wi \Lambda_1}$ is  the identity map on $\Wi \Lambda_1$.
If $\Wi x \in \Wi N_1 - \Wi \Lambda_1$, there exists a unique point $\Wi y \in \Wi T_1$ and a unique $t\in \RR^+$ so that
$\Wi x = \Wi Y_t^1 (\Wi y)$.
In this case, $\Wi H_1 (\Wi x)$ is defined by:
\begin{enumerate}
  \item $\Wi Y_{\frac{K+ \Wi g_1 (\Wi y)}{K} t}^1 \circ \Wi h_1 (\Wi y)$, when $t\in [0,K]$;
  \item $\Wi Y_{t+ \Wi g_1 (\Wi y)}^1 \circ \Wi h_1 (\Wi y)$, when $t\in [K, +\infty)$.
\end{enumerate}
$\Wi H_1$ satisfies the following  commutative properties:

\begin{lemma}\label{l.cH1}
\begin{enumerate}
  \item $\Wi H_1$ is commutative to $\sigma_1$: $\Wi H_1 \circ \sigma_1 (\Wi x)= \sigma_1 \circ \Wi H_1 (\Wi x)$ for every point $\Wi x \in \Wi N_1$.
  \item  If $\Wi x \in \Wi Y_K^1 (\Wi N_1)- \Wi \Lambda_1$, then $\Wi P_1 \circ \Wi H_1 (\Wi x)= \Wi P_1 (\Wi x)$.
\end{enumerate}

\end{lemma}
\begin{proof}
First we prove item $1$.
If $\Wi x \in \Wi \Lambda_1$, the commutative equality can be immediately
followed by the fact that $\Wi H_1 |_{\Wi \Lambda_1}=Id$.

If $\Wi x \in \Wi N_1 - \Wi \Lambda_1$,
$\Wi x = \Wi Y_t^1 (\Wi y) $ for some $\Wi y \in \Wi T_1$ and some $t\geq 0$.
Then we divide the proof in two subcases: $t \leq K$ and $t\geq K$.

When $t \leq K$,
\begin{eqnarray*}
   \Wi H_1 \circ \sigma_1 (\Wi x) &=& \Wi Y_{\frac{K+ \Wi g_1 \circ\sigma_1 (\Wi y)}{K} t}^1 \circ \Wi h_1 \circ \sigma_1 (\Wi y)\nonumber\\
                           &=& \Wi Y_{\frac{K+ \Wi g_1  (\Wi y)}{K} t}^1 \circ \sigma_1 \circ\Wi h_1 (\Wi y)\nonumber\\
                           &=&\sigma_1 \circ \Wi Y_{\frac{K+ \Wi g_1  (\Wi y)}{K} t}^1 \circ \Wi h_1 (\Wi y)\nonumber\\
                           &=& \sigma_1 \circ \Wi H_1 (\Wi x).
\end{eqnarray*}

When $t\geq K$,
\begin{eqnarray*}
   \Wi H_1 \circ \sigma_1 (\Wi x) &=& \Wi Y_{t+ \Wi g_1 \circ\sigma_1(\Wi y)}^1 \circ \Wi h_1 \circ \sigma_1(\Wi y)\nonumber\\
                           &=& \Wi Y_{t+ \Wi g_1 (\Wi y)}^1 \circ \sigma_1\circ\Wi h_1  (\Wi y)\nonumber\\
                           &=&\sigma_1\circ\Wi Y_{t+ \Wi g_1 (\Wi y)}^1 \circ \Wi h_1  (\Wi y)\nonumber\\
                           &=& \sigma_1 \circ \Wi H_1 (\Wi x).
\end{eqnarray*}

The second equality in each  recursive equation above  is followed by item $3$ of Lemma \ref{l.ch1}
and item $1$ of Proposition \ref{p.cg1}, and the third equality in each  recursive equation above is followed by  item $3$ of Lemma \ref{l.comm}.

Now we turn to prove item $2$.
$\Wi x \in \Wi Y_K^1 (\Wi N_1)- \Wi \Lambda_1$,
then there exists a unique point $\Wi y \in \Wi T_1$ so that $\Wi x= \Wi Y_t^1 (\Wi y)$ for some $t\geq K$.
Since $t\geq K$, $\Wi P_1 \circ \Wi H_1 (\Wi x)= \Wi P_1 \circ \Wi Y_{t+ \Wi g_1 (\Wi y)}^1 \circ \Wi h_1 (\Wi y)= \Wi P_1 \circ \Wi h_1 (\Wi y) + t+ \Wi g_1 (\Wi y)$.
 $\Wi P_1 \circ  (\Wi x)= \Wi P_1 \circ \Wi Y_t (\Wi y)= \Wi P_1 (\Wi y) +t$.
 Recall that  $\Wi g_1 (\Wi x)= \Wi P_1 (\Wi x)-\Wi P_1 \circ \Wi h_1 (\Wi x)$,
 then  $\Wi P_1 \circ \Wi H_1 (\Wi x)= \Wi P_1 (\Wi x)$. Item $2$ is proved.
\end{proof}

\begin{lemma}\label{l.samesach}
 If $\Wi x \in \Wi Y_K^1 (\Wi N_1)- \Wi \Lambda_1$, then $\Wi x$ and $\Wi H_1 (\Wi x)$ are in the same strong s-arch in $\Wi N_1 - \Wi \Lambda_1$.
\end{lemma}
\begin{proof}
$\Wi x \in \Wi Y_K^1 (\Wi N_1)- \Wi \Lambda_1$,
then there exists a unique point $\Wi y \in \Wi T_1$ so that $\Wi x= \Wi Y_t^1 (\Wi y)$ for some $t\geq K$.
In this case, $\Wi H_1 (\Wi x)=\Wi Y_{t+ \Wi g_1 (\Wi y)}^1 \circ \Wi h_1 (\Wi y)$.
  Recall  that $\Wi h_1 (\Wi y)$  and $\Wi y$ are in the same leaf $\Wi l^s$ of $\Wi F^s$,
 and
 $\Wi H_1 (\Wi x)$ and $\Wi x$ are in the orbits of $\Wi Y_t$ starting at $\Wi h_1 (\Wi y)$  and $\Wi y$ respectively.
 Assume that $\Wi L^s$ is the  leaf in $\Wi {\cF}^s$ which contains $\Wi l^s$.
 Since $\Wi l^s \subset \Wi L^s - \Wi \Lambda_1$ and $\Wi l^s$ is connected,
then $\Wi l^s$  belongs to a connected plague $ D_{\Wi l^s}$ of $\Wi L^s - \Wi \Lambda_1$.
By item $2$ of Lemma \ref{l.cH1}, both of $\Wi x$ and $\Wi H_1 (\Wi x)$ are in the same section $\Wi \Sigma_{\Wi P_1 (\Wi x)}$.
Therefore, $\Wi x$ and $\Wi H_1 (\Wi x)$ are in $D_{\Wi l^s} \cap\Wi \Sigma_{\Wi P_1 (\Wi x)}$.
So, we need to understand $D_{\Wi l^s} \cap\Wi \Sigma_{\Wi P_1 (\Wi x)}$ more clear.

There are two possibilities    for the intersection of $\Wi \Lambda_1$ and the closure of $D_{\Wi l^s}$:
\begin{enumerate}
  \item one orbit $\Wi \gamma$  which is the lift of a
boundary periodic orbit of $\Lambda_1$;
  \item the union of two orbits $\Wi \gamma_1$ and $\Wi \gamma_2$ in the two adjacent boundary leaves of $\Wi \Lambda_1$.
\end{enumerate}

In the first case, $D_{\Wi l^s} \cap\Wi \Sigma_{\Wi P_1 (\Wi x)}$ is a strong s-arch $a_p$ with two ends in the lift of a compact leaf of $F^s$ and $\Wi \gamma$ respectively. Then the conclusion of the lemma is followed in this case.

There are two sub-cases for $D_{\Wi l^s} \cap\Wi \Sigma_{\Wi P_1 (\Wi x)}$ in the second case.
The first sub-case is that $D_{\Wi l^s} \cap\Wi \Sigma_{\Wi P_1 (\Wi x)}$ is a  strong s-arch $a_b$ with two ends in $\Wi \gamma_1$ and $\Wi \gamma_2$ respectively. Then the conclusion of the lemma  is followed in this sub-case.
The second sub-case is that $D_{\Wi l^s} \cap\Wi \Sigma_{\Wi P_1 (\Wi x)}$ is
the union of two strong s-arches $a_b^1$  and $a_b^2$ so that:
 \begin{enumerate}
   \item  one end of $a_b^i$ ($i\in \{1,2\}$) is in $\Wi \gamma_i$ and the other end  is in $\Wi l^s$;
   \item $\Wi Y_t^1 (a_b^1 \sqcup a_b^2)$ is in a strong s-arch  of $\Wi N_1 - \Wi \Lambda_1$ when $t$ is big enough.
 \end{enumerate}
 See Figure \ref{f.sarch} as the illustrations of some notations above.
 In the second case, it is possible that $\Wi x$ and $\Wi H_1 (\Wi x)$ are not in the same strong s-arch, i.e.
 $\Wi x \in a_b^1 $ and $\Wi H_1 (\Wi x) \in a_b^2$, or $\Wi x \in a_b^2 $ and $\Wi H_1 (\Wi x) \in a_b^1$. In the following, we will finish the proof of the lemma by showing that this  never happens.

 We define an open connected  subset $E$ (see Figure \ref{f.sarch}) of $\Sigma_{\Wi P_1 (\Wi x)}$ by the union  of the strong stable-arches so that one end of each arch is in $W^u(\Wi \gamma)$ and the other end is in
 $\Wi T^1$.  We define by $E_0$ the intersection of $\Wi Y_K^1 (\Wi N_1)$ and $E$. One can routinely
 check that $E_0$ also is connected.
 We define by $E_1$ the union of the point $\Wi x$ in $E_0$ which satisfies that  $\Wi H_1 (\Wi x) $
 is not in $E$. Note that  when $\Wi x \in E_1$,  $\Wi H_1 (\Wi x) $ is in the interior of  a strong
 s-arch which is disjoint to $E$. This fact and the continuity of  $\Wi H_1$ ensure that $E_1$ is an open subset of
 $E$. Further notice that every point $\Wi x \in E_0-E_1$ satisfies that $\Wi H_1 (\Wi x)$ is in the open set $E$.
 Therefore, $E_0-E_1$ also is an open subset of $E_0$.
 Since $E_0$ is connected, then either $E_1$ or $E_0-E_1$ is empty.
 But obviously $a_p \cap E_1-E_0$ is non-empty, hence $E_1$ has to be  empty.
 This means that the phenomena in the end of the last paragraph never happens.
 The lemma is proved.
\end{proof}

\begin{figure}[htp]
\begin{center}
  \includegraphics[totalheight=7cm]{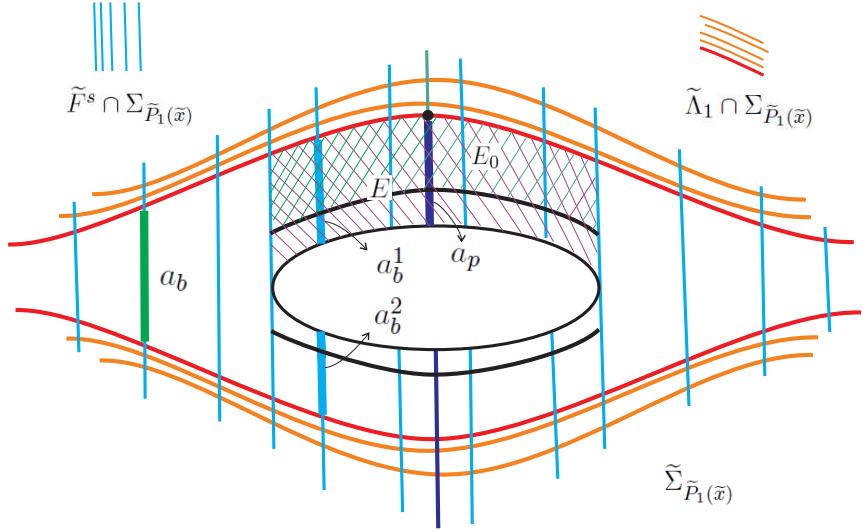}\\
  \caption{some notations on $\Sigma_{\Wi P_1 (\Wi x)}$}\label{f.sarch}
\end{center}
\end{figure}

The following proposition about $\Wi H_1$ is important to show the continuity of $\Wi H_1$ on $\Wi N_1$.

\begin{proposition}\label{p.cH1}
If $\lim_{n\to \infty} \Wi x_n =\Wi x \in \Wi \Lambda_1$,  then $\lim_{n\to \infty} d(\Wi x_n, \Wi H_1 (\Wi x_n))=0$ where the function $d(\Wi x, \Wi y)$ means the distance of $\Wi x$ and $\Wi y$ under some fixing Riemannian metric on $\Wi N_1$.
\end{proposition}
\begin{proof}
 Since $\Wi H_1 \mid_{\Wi \Lambda_1}=Id$ and $\lim_{n\to \infty} \Wi x_n =\Wi x \in \Wi \Lambda_1$
  we only need to focus on the case that $\Wi x_n \in \Wi Y_K^1 (\Wi N_1)- \Wi \Lambda_1$.
  Then $\Wi x_n= \Wi Y_{K+t_n}^1 (\Wi y_n)= \Wi Y_{t_n}^1\circ\Wi Y_K^1 (\Wi y_n)$ for  some $y_n \in \Wi T_1$ and some $t_n\geq 0$. Here $\{t_n\}$ satisfies that $\lim_{n\to \infty} t_n =+\infty$.
   By the definition of $\Wi H_1$, we have
  $\Wi H_1 (\Wi x_n)= \Wi Y_{t_n}^1\circ\Wi Y_{K+\Wi g_1 (\Wi y_n)}^1 \circ \Wi h_1(\Wi y_n)$.
  Note that $\Wi H_1 \circ \Wi Y_K^1 (\Wi y_n)= \Wi Y_{K+\Wi g_1 (\Wi y_n)}^1 \circ \Wi h_1(\Wi y_n)$.

  For simplicity, set $\Wi z_n = \Wi Y_K^1 (\Wi y_n)$ and $\Wi w_n =\Wi Y_{K+\Wi g_1 (\Wi y_n)}^1 \circ \Wi h_1(\Wi y_n)$. Then $\Wi x_n= \Wi Y_{t_n}^1(\Wi z_n)$, $\Wi H_1 (\Wi x_n)= \Wi Y_{t_n}^1(\Wi w_n)$ and
  $\Wi H_1 (\Wi z_n) = \Wi w_n$.
  If we endow $N_1$ a Riemannian metric, then it naturally induces a lifting metric on $\Wi N_1$.
  By the construction of  a standard DA attractor $(N_0, Y_t^0, \Lambda_0)$ introduced in Section \ref{sss.epaflow}, the lengths of the strong s-arches in $N_0 -\Lambda_0$  is uniformly bounded.
  Further notice the facts that  $(N_1, Y_t^1, \Lambda_1)$  is a copy of $(N_0, Y_t^0, \Lambda_0)$ and
  $(\Wi N_1, \Wi Y_t^1, \Wi \Lambda_1)$ ia a cyclic cover of $(N_1, Y_t^1, \Lambda_1)$.
  Therefore, the  lengths of the strong s-arches of $\Wi N_1 - \Wi \Lambda_1)$  also is uniformly bounded.
  By Lemma \ref{l.samesach}, $\Wi z_n$ and $\Wi w_n=\Wi H_1 (\Wi z_n)$ are in the same strong s-arch in $\Wi N_1 - \Wi \Lambda_1$.
Then  there exists $K_1 >0$ so that  $d_s(\Wi z_n, \Wi w_n) <K_1$ for every $n\in \NN$.
  Here $d_s$ is the distance along the corresponding strong s-arch.
  Since $\Wi Y_t^1$ uniformly exponentially contracts on every strong stable manifold and  $\lim_{n\to \infty} t_n =+\infty$,
  we have  the distance of $\Wi x_n= \Wi Y_{t_n}^1(\Wi z_n)$ and $\Wi H_1 (\Wi x_n)= \Wi Y_{t_n}^1(\Wi w_n)$ in $\Wi N_1$ is asymptotic to $0$ when $n\to \infty$.
 \end{proof}

The following lemma is  a key lemma to prove Theorem \ref{t.claAno}.

\begin{lemma}\label{l.key}
$\Wi H_1: \Wi N_1 \to \Wi N_1$  is a homeomorphism which preserves the orbits of $\Wi Y_t^1$.
\end{lemma}

\begin{proof}
By the definition of $\Wi H_1$, it is easy to check that $\Wi H_1: \Wi N_1 - \Wi \Lambda_1 \to \Wi N_1 -\Wi \Lambda_1$ is a homeomorphism which preserves the orbits of $\Wi Y_t^1$.
Note that $\Wi H_1 |_{\Wi \Lambda_1}$ is the identity map on $\Wi \Lambda_1$.
Therefore, $\Wi H_1$ is a bijection map on $\Wi N_1$ which preserves the orbits of $\Wi Y_t^1$.
We are left to check that both of $\Wi H_1$ and $\Wi H_1^{-1}$ are continuous.
Since $\Wi H_1 \circ \sigma_1 = \sigma_1 \circ \Wi H_1 $ (Lemma \ref{l.cH1}),
$\Wi H_1$ can induces a bijection self map on the compact manifold $N_1$ (which we will talk more in the next corollary). This means that the continuity of $\Wi H_1$ can implies the continuity of
$\Wi H_1^{-1}$. Therefore, we only need to focus on showing that $\Wi H_1$ is continuous.

$\Wi H_1$ obviously is continuous on the open set $\Wi N_1 - \Wi \Lambda_1$,
thus we only need to check the continuity on $\Wi \Lambda_1$.
Set $\Wi x \in \Wi \Lambda_1$ and $\{\Wi x_n\} \subset \Wi N_1$ so that $\lim_{n\to \infty} \Wi x_n= \Wi x$. By item $2$ of Lemma \ref{p.cH1},  we have $\lim_{n\to \infty} d(\Wi x_n, \Wi H_1 (\Wi x_n))=0$.
Then $\lim_{n\to \infty} \Wi H_1 (\Wi x_n)=\Wi x$. This means that $\Wi H_1$ is continuous on any point $\Wi x \in \Wi \Lambda_1$.
\end{proof}

\subsection{End of the proof of Theorem \ref{t.claAno}}\label{ss.ep1}
\begin{corollary}\label{c.H1}
$\Wi h_1$ and $\Wi H_1$ can induce two homeomorphisms $h_1$ and $H_1$ on $T_1$ and $N_1$
respectively so that:
\begin{enumerate}
  \item $H_1 \circ \pi_1 = \pi_1 \circ \Wi H_1$ and $h_1 \circ \pi_1 = \pi_1 \circ \Wi h_1$;
  \item $H_1: N_1 \to N_1$ is an extension of $h_1: T_1 \to T_1$;
  \item $H_1 |_{\Lambda_1}$ is the identity map on $\Lambda_1$;
  \item $H_1: N_1 \to N_1$ preserves the orbits of $Y_t^1$.
\end{enumerate}
\end{corollary}

\begin{proof}
Since  $\Wi H_1$ is a homeomorphism on $\Wi N_1$ commutative to $\sigma_1$ (Lemma \ref{l.key} and Lemma \ref{l.cH1}),
 $\Wi H_1$ induces a homeomorphism $H_1$ on $N_1$ so that $H_1 \circ \pi_1 = \pi_1 \circ \Wi H_1$.
Recall that $\Wi H_1 |_{\Wi \Lambda_1}$ is the identity map on $\Wi \Lambda_1$,
therefore $H_1 |_{\Lambda_1}$ is the identity map on $\Lambda_1$.
  Note that $\Wi H_1$  preserves the orbits of $\Wi Y_t^1$  (Lemma \ref{l.key}), then $H_1$ preserves the orbits of $Y_t^1$.

  Since  $\Wi h_1$ is a homeomorphism  on $\Wi T_1$ commutative to $\sigma_1$ (Lemma \ref{l.ch1}),
  $\Wi h_1$ induces a homeomorphism $h_1$ on $N_1$ so that $h_1 \circ \pi_1 = \pi_1 \circ \Wi h_1$.
  The conclusion that $H_1$ is an extension of $h_1$ can be immediately
  followed by the fact that $\Wi H_1: \Wi N_1 \to \Wi N_1$ is an extension of $\Wi h_1: \Wi T_1 \to \Wi T_1$.
 \end{proof}

We can do a similar arguments for $(N_2, Y_{-t}^2)$ based on $\Wi h_2$ on $\Wi T_2$ (Corollary \ref{c.ch2}) to get two homeomorphisms $h_2$ on $T_2$
and $H_2$ on $N_2$ as the following lemma.

\begin{lemma}\label{l.H2}
There exist two homeomorphisms $h_2$ on $T_2$
and $H_2$ on $N_2$ which satisfy the following conditions:
\begin{enumerate}
  \item $h_2$ satisfies the commutative equality: $h_2 \circ \pi_2 = \pi_2 \circ \Wi h_2$;
  \item $H_2: N_2 \to N_2$ is an extension of $h_2: T_2 \to T_2$;
  \item $H_2 |_{\Lambda_2}$ is the identity map on $\Lambda_2$;
  \item $H_2: N_2 \to N_2$ preserves the orbits of $Y_{-t}^2$.
\end{enumerate}
\end{lemma}

 Now we can finish the proof of Theorem \ref{t.claAno}

 \begin{proof}[Proof of Theorem \ref{t.claAno}]
 We will prove Theorem \ref{t.claAno} by proving Claim \ref{cl.final}.
 Pick a point $x\in T_2$ and let $\Wi x$ be a point in $\Wi T_2$ so that $\pi_2 (\Wi x)= x$.
By item $1$ of Corollary \ref{c.H1} and item $5$ of Lemma \ref{l.comm}, $\Psi_1 \circ \pi_2 = \pi_1 \circ \Wi \Psi_1$ and
 $h_1 \circ \pi_1 = \pi_1 \circ \Wi h_1$. Then one can easily get that $h_1 \circ \Psi_1 (x)= \pi_1 \circ \Wi h_1 \circ \Wi \Psi_1 (\Wi x)$ by tracking the two commutative equalities above.
 Similarly, one can check that $h_2 \circ \Psi_2 (x)= \pi_1  \circ \Wi \Psi_2 \circ \Wi h_2 (\Wi x)$.
By the definition of $\Wi h_2$ (Corollary \ref{c.ch2}),  $\Wi h_1 \circ \Wi \Psi_1 = \Wi \Psi_2 \circ \Wi h_2$. Therefore, $h_1 \circ \Psi_1 (x)= \Psi_2 \circ h_2 (x)$.

Recall that $M_i= N_1 \sqcup N_2 / x\sim \Psi_i (x), x\in T_2$ ($i\in \{1,2\}$), by the facts
$H_1 |_{T_1}=h_1$ (Corollary \ref{c.H1}), $H_2 |_{T_2}=h_2$ (Lemma \ref{l.H2}) and $h_1 \circ \Psi_1 = \Psi_2 \circ h_2 $,
then we can compose $H_1$ and $H_2$ to a homeomorphism $H: M_1 \to M_2$ so that $H|_{N_1}= H_1$ and
$H|_{N_2}= H_2$.
Moreover, since $H_1$ preserves the orbits of $Y_t^1$ and
$H_2$ preserves the orbits of $Y_{-t}^2$, $H$ maps each orbit of $Z_t^1$ on $M_1$ to a orbit of $Z_t^2$ on  $M_2$. Therefore, $(M_1, Z_t^1)$ and $(M_2, Z_t^2)$ are topologically equivalent.
Claim \ref{cl.final} and Theorem \ref{t.claAno} are proved.
 \end{proof}

\section{Generalizations}\label{s.final}

This section  will mainly  discuss the two generalization questions in Question \ref{q.gen}.
We  will also give some comments for classifying expanding attractors on fibered hyperbolic $3$-manifolds.

\subsection{A sketch of the proof of Theorem \ref{t.gFW1}}
Recall that $M_B$ is  obtained by gluing two copies of $N_0$, say $N_1$ and $N_2$,
through a gluing automorphism $B\in SL(2, \ZZ)$.
When  either
$B=\left(
                                    \begin{array}{cc}
                                      1 & 0 \\
                                      k & 1 \\
                                    \end{array}
                                  \right)
  $ or
  $B=\left(
                                    \begin{array}{cc}
                                      -1 & 0 \\
                                      k & -1 \\
                                    \end{array}
                                  \right)
  $ for some $k\in \ZZ$,
 there exists a non-transitive Anosov flow $Z_t^B$  on $M_B$ which is our model flow in Theorem \ref{t.gFW1}.
Now let us introduce the model flow $Z_t^B$ more precisely. Note that in the two cases above,
$B (F^u)$ can be transverse to $F^s$ everywhere in $\partial N_1$. Using the strategy in \cite{FW} or Theorem 1.5 of \cite{BBY},
  there exists a homeomorphism $\Psi_B: \partial N_2 \to \partial N_1$ which is isotopic to $B$
 so that the glued flow is a non-transitive Anosov flow on the glued manifold, which is homeomorphic
 to $M_B$. Therefore, up to topological equivalence, we can think of this flow as a non-transitive Anosov flow
 on the manifold $M_B$, which we denote by $Z_t^B$.
 We remark that since $\left(\begin{array}{cc}
                                                  2 & 1 \\
                                                  1 & 1
                                                \end{array}\right)$ and
                                                $\left(\begin{array}{cc}
                                                  2 & 1 \\
                                                  1 & 1
                                                \end{array}\right)^{-1}$
                                                are conjugate by
                                                $\left(\begin{array}{cc}
                                                  -1 & 0 \\
                                                  0 & -1
                                                \end{array}\right)$,
                                                 $Z_t^{B_1}$ and
                                                $Z_t^{B_2}$  are topologically equivalent.
                                                Here $B_1=\left(\begin{array}{cc}
                                                  1 & 0 \\
                                                  k & 1
                                                \end{array}\right)$ and
                                                $B_2=\left(\begin{array}{cc}
                                                  -1 & 0 \\
                                                  -k & -1
                                                \end{array}\right)$.

Now we start to prove Theorem \ref{t.gFW1}. But we only give a sketch of the proof since most parts will
be identical to the proof of Theorem \ref{t.claAno}.

\begin{proof}[A brief proof of Theorem \ref{t.gFW1}]
Let $X_t$ be a non-transitive Anosov flow on $M_B$.
By using the same proof to Lemma  \ref{l.decomp},
one  can get that $X_t$ shares a similar decomposition structure  to
nontransitive Anosov flow on  $M$:
$X_t$ is topologically equivalent to the flow obtained
by gluing  $(N_1, Y_t^1)$ and
$(N_2, Y_{-t}^2)$ together through a homeomorphism $\Psi: \partial N_2 \to \partial N_1$
so that:
\begin{enumerate}
  \item $\Psi$ is isotopic to $B$;
  \item $\Psi (F^u)$ and $F^s$ are transverse in the glued torus $T$.
\end{enumerate}
Note that every notation which is not defined here  is the same to the proof of Lemma  \ref{l.decomp}.

When $B$ is neither   $\left(
                                                                            \begin{array}{cc}
                                                                              1 & 0 \\
                                                                              k & 1 \\
                                                                            \end{array}
                                                                          \right)$
nor  $\left(
       \begin{array}{cc}
         -1 & 0 \\
         k & -1 \\
       \end{array}
     \right)$ for some $k\in \ZZ$, we
     pick a compact leaf $c^u$ in $F^u$,
     then $\Psi(c^u)$ is not isotopic to a compact leaf of $F^s$ on $\partial N_1$.
     Hence,
     $\Psi(c^u)$ has to transversely pass through
     the two Reeb annuli of $F^s$.  This certainly is impossible.
     Therefore, in this case, there does not exist any non-transitive Anosov flow on $M_B$ at all. Item $2$ of the theorem is proved.

  When either $B=\left(
                                                                            \begin{array}{cc}
                                                                              1 & 0 \\
                                                                              k & 1 \\
                                                                            \end{array}
                                                                          \right)$
or  $B=\left(
       \begin{array}{cc}
         -1 & 0 \\
         k & -1 \\
       \end{array}
     \right)$ for some $k\in \ZZ$, one can  check that every step in Section \ref{s.nonAno}
still works  and we can similarly use them to prove that $X_t$ is topologically equivalent to
$Z_t^B$. The only  difference is that now the gluing map $\Psi$ is not isotopic to $Id$. But this difference does not affect any step of the proof.  Item $1$ of the theorem is proved.
\end{proof}

\subsection{Expanding attractors on $N_A$ and non-transitive Anosov flows on $M_A^B$}
In this subsection,
we turn to item $2$ of Question \ref{q.gen} regarding the classifications of expanding attractors supported by $N_A$ and non-transitive Anosov flows on $N_A^B$, starting with the discussion of the first.

There are two DA expanding attractors
$\Lambda_A$ and $\Lambda_{A^{-1}}$ supported by $N_A$ under the two flows $Y_t^A$ and $Y_t^{A^{-1}}$.\footnote{The definitions of these notions can be found in Section \ref{ss.Ge}.}
Note that different from the case $A=\left(\begin{array}{cc}
                                                  2 & 1 \\
                                                  1 & 1
                                                \end{array}\right)$,
generally $Y_t^A$ and $Y_t^{A^{-1}}$  are not topologically equivalent since generally
$A$ and $A^{-1}$ are not conjugate in $GL({2,\ZZ})$.
One can  consider classifying expanding attractors on $N_A$
in a similar way to the proof of Theorem \ref{t.claexp}.
After running all the steps of the proof, it is easy to find that
the only gap is that we can not build Proposition \ref{p.no0} by using a parallel proof.

Let us explain in detail.
Let $Y_t$ be a smooth flow on $N_A$ so that the maximal invariant set of $Y_t$
is an expanding attractor $\Lambda$ supported by $N_A$.
Recall that $\cF^s$ is the stable foliation of the attractor  and $F^s =\cF^s \cap T$ ($T=\partial N_A$)
contains $p$ compact leaves $c_1,\dots, c_p$. Proposition \ref{p.no0}  tells us that
$c_i$ ($i\in \{1,\dots, p\}$) is not parallel to $c_h$ in $T$ in the case $A= \left(\begin{array}{cc}
                                                  2 & 1 \\
                                                  1 & 1
                                                \end{array}\right)$ and $N_A =N_0$.

Let $O_A$ be the manifold obtained by  filling a solid torus $V$ to $N_A$ so that a meridian of $V$ is glued to $c_v$.
If $p$ is even, similar to the constructions in the proof of Proposition \ref{p.no0}, one can build a taut foliation in  $O_A$.
If $p$ is odd, one can do a double cover of $N_A$ associated to the once-punctured torus fibration structure on $N_A$. In fact, the covering manifold  is homeomorphic to $N_{A^2}$. Then we can similarly build a taut foliation in  $O_{A^2}$ by extending the double cover of $\Lambda$.
In the case that $A=\left(\begin{array}{cc}
                                                  2 & 1 \\
                                                  1 & 1
                                                \end{array}\right)$,
$O_A$ is homeomorphic to $S^3$ and $O_{A^2}$ is homeomorphic to the lens space $L(5,2)$. Then we can get a contraction
by the fact that neither $O_A$ nor $O_{A^2}$  carry any taut foliation.

To fill the gap in general case, it is natural to ask the following topological question: does there exist any taut foliation on $O_A$? As far as we know, the best result on this topic can be obtained by  Theorem 4.1 of Baldwin \cite{Bal}:  $O_A$ does not carry any co-orientable taut foliation.\footnote{Theorem 4.1 of \cite{Bal}  tells us that $O_A$ is a $L$-space. It is well-known for low-dimensional topologists that an $L$-space never carries any co-orientable taut foliation.} We say that the attractor $\Lambda$ is \emph{co-orientable} if $\Lambda$ is
a co-orientable lamination in $N_A$.
By tracking our construction, one can find that
the constructed foliation is co-orientable if and only if $\Lambda$ is co-orientable.
Then, this means that $N_A$ does not support a co-orientable expanding attractor so that its boundary periodic orbits
are parallel to $c_h$.

We can summarize these discussions to the following theorem:

\begin{theorem}\label{t.claexpA}
Up to topological equivalence, every
flow  $Y_t$ with expanding attractor $\Lambda$ supported by
$N_A$ is topologically equivalent to either
$Y_t^A$ or $Y_t^{A^{-1}}$
except for the case that the boundary periodic orbits
are parallel to $c_h$ and $\Lambda$ is not co-orientable.
In particular,  every
flow  with co-orientable expanding attractor supported by
$N_A$ is topologically equivalent to either
$Y_t^A$ or $Y_t^{A^{-1}}$
\end{theorem}

The techniques in Section \ref{s.nonAno} still work for the manifold $M_B^A$. Therefore,  we can describe
most non-transitive Anosov flows on $M_B^A$ based on Theorem \ref{t.claexpA}. We would not like to state such a theorem
here since it is  incomplete and too subtle. Instead, we would like to ask the following topological question
whose negative answer will fill our gap above.

\begin{question}\label{q.final}
Does there exist any non-co-orientable taut foliation on $O_A$?
\end{question}

\subsection{Comments about expanding attractors on fibered hyperbolic $3$-manifolds}
It is interesting to classify expanding attractors on any given fibered hyperbolic
$3$-manifold  $N$ whose boundary is the union of finitely many tori.
The canonical example is DPA expanding attractor, which is derived from suspension pseudo-Anosov
flow. One naturally expects this class of expanding attractors will be the model attractors
in our question. We would like to provide two comments about it.
\begin{itemize}
  \item Let $\Lambda$  be  an expanding attractor supported by $N$.
Suppose that   every periodic orbit in $\Lambda$ algebraically intersects with a fiber surface of
some surface fibration structure on $N$.
Using the strategies developed in Section \ref{s.claexp}, one can show that $\Lambda$ is a DPA expanding attractor.
  \item
  Due to the first comment, it is natural to ask the following question.
Let $\Lambda$  be  an expanding attractor supported by  $N$.
Does there exist  a fibration structure $F$ on $N$ so that
every periodic orbit in $\Lambda$  algebraically intersects with a fiber surface of $F$?
In the preparing work \cite{BY2},
the authors will construct an example to negatively  answer the question above.
Therefore, it seems to be far from reached at present  for a complete classification of  expanding attractors on fibered hyperbolic $3$-manifolds.
\end{itemize}

\vskip 1cm
\noindent Jiagang Yang

\noindent {\small Departamento de Geometria, Instituto de Matem\'atica e Estatistica}

\noindent{\small Universidade Federal Fluminense, Niter\'oi, BRAZIL}

\noindent{\footnotesize{E-mail: yangjg@impa.br}}
\vskip 2mm

\noindent Bin Yu

\noindent {\small School of Mathematical Sciences}

\noindent{\small Tongji University, Shanghai 200092, CHINA}

\noindent{\footnotesize{E-mail: binyu1980@gmail.com }}


\begin{thebibliography}{MM}
\bibitem[An]{An}
Anosov, D. V.
\emph{Geodesic flows on closed Riemannian manifolds of negative curvature.}
Trudy Mat. Inst. Steklov.  90  (1967).
\bibitem[Bal]{Bal} Baldwin, John A. \emph{Heegaard Floer homology and genus one, one-boundary component open books.} J. Topol. 1 (2008), no. 4, 963-992.
\bibitem[Bar]{Bar}
Barbot, Thierry. \emph{Generalizations of the Bonatti-Langevin example of Anosov flow and their classification up to topological equivalence.} Comm. Anal. Geom.  6  (1998),  no. 4, 749-798.
\bibitem[Bart]{Bart} Thomas Barthelme. \emph{School on contemporary dynamical systems: Anosov
flows in dimension $3$.} 2017.
\bibitem[BB]{BB}  Beguin, Francois;  Bonatti, Christian. \emph{ Flots de Smale en dimension 3: pr\'esentations finies de voisinages invariants d¡¯ensembles selles.} Topology 41 (2002) 119-162
\bibitem[BBY]{BBY}  Beguin, Francois;  Bonatti, Christian and Yu, Bin.  \emph{Building Anosov flows on $3$-manifolds}  Geom. Topol.  21  (2017),  no. 3, 1837-1930.
 \bibitem[BF1]{BF1}   Barbot, Thierry; Fenley, S\'ergio \emph{Pseudo-Anosov flows in toroidal manifolds.} Geom. Topol. 17 (2013) 1877-1954
  \bibitem[BF2]{BF2} Barbot, Thierry; Fenley, S¨¦rgio R. \emph{Classification and rigidity of totally periodic pseudo-Anosov flows in graph manifolds.} Ergodic Theory Dynam. Systems 35 (2015), no. 6, 1681-1722.
\bibitem[BL]{BL}    Bonatti, C; Langevin,R \emph{Un exemple de flot d¡¯Anosov transitif transverse a un tore et
non conjugue a une suspension.} Ergodic Theory Dynam. Systems 14 (1994) 633-643
\bibitem[BLJ]{BLJ} Bonatti, C.; Langevin, R. \emph{Diff\'eomorphismes de Smale des surfaces. (French) [Smale diffeomorphisms of surfaces]} With the collaboration of E. Jeandenans. Ast\'erisque No. 250 (1998), viii+235 pp.
\bibitem[BM]{BM} Bowden, Jonathan  and  Mann, Kathryn. \emph{$C^0$ stability of boundary actions and inequivalent Anosov flows.}  arXiv:1909.02324, 2019
\bibitem[BR]{BR} Brittenham, Mark; Roberts, Rachel. \emph{When incompressible tori meet essential laminations.} Pacific J. Math. 190 (1999), no. 1, 21-40.
\bibitem[Br]{Br}  Brunella, Marco. \emph{Separating the basic sets of a nontransitive Anosov flow.} Bull. London Math. Soc. 25 (1993), no. 5, 487-490.
\bibitem[BW]{BW}    Birman, Joan S.; Williams, R. F. \emph{Knotted periodic orbits in dynamical system. II. Knot holders for fibered knots.} Low-dimensional topology (San Francisco, Calif., 1981), 1-60, Contemp. Math., 20, Amer. Math. Soc., Providence, RI, 1983.
\bibitem[BY1]{BY1}   Francois, Beguin; Yu, Bin.   \emph{A uniqueness theorem for transitive Anosov flows obtained by gluing hyperbolic plugs.} arXiv:1905.08989, 2019
\bibitem[BY2]{BY2}   Francois, Beguin; Yu, Bin.   \emph{Coding orbits of Anosov flows obtained by gluing hyperbolic
plugs: coding periodic orbits and its applications.} in preparation
 \bibitem[Cal]{Cal} Calegari, Danny. \emph{Foliations and the geometry of 3-manifolds.} Oxford Mathematical Monographs. Oxford University Press, Oxford, 2007. xiv+363 pp.
 \bibitem[Ch1]{Ch1}    Christy, Joe. \emph{Anosov flows on three-manifolds (Topology, Dynamics).}
 Thesis (Ph.D.)-University of California, Berkeley. 1984. 83 pp
 \bibitem[Ch2]{Ch2} Christy, Joe. \emph{Branched surfaces and attractors. I. Dynamic branched surfaces.} Trans. Amer. Math. Soc. 336 (1993), no. 2, 759-784.
 \bibitem[Fen1]{Fen1} Fenley, S\'ergio R. \emph{Anosov flows in 3-manifolds.} Ann. of Math. 139 (1994) 79-115
\bibitem[Fen2]{Fen2} Fenley, S\'ergio R. \emph{Quasigeodesic Anosov flows and homotopic properties of flow lines} J. Differential Geom.  41  (1995),  no. 2, 479-514.
\bibitem[FLP]{FLP}  Fathi, Albert; Laudenbach, Francois; Po\'enaru, Valentin \emph{Thurston's work on surfaces} Translated from the 1979 French original by Djun M. Kim and Dan Margalit. Mathematical Notes, 48. Princeton University Press, Princeton, NJ, 2012. xvi+254 pp.
 \bibitem[Fra]{Fra}    Franks, John \emph{Anosov diffeomorphisms.} 1970 Global Analysis (Proc. Sympos. Pure Math., Vol. XIV, Berkeley, Calif., 1968) pp. 61-93 Amer. Math. Soc., Providence, R.I.
\bibitem[Fri]{Fri} Fried, David. \emph{Fibration over $S^1$ with pseudo-Anosov monodromy} in [FLP], expose 14.
\bibitem[Fu]{Fu} Fuller, F. Brock. \emph{On the surface of section and periodic trajectories} Amer. J. Math.  87  (1965) 473-480.
\bibitem[FW]{FW} Franks, J;  Williams, B \emph{Anomalous Anosov flows.} from ¡°Global theory of dynamical
systems¡± (Z Nitecki, C Robinson, editors), Lecture Notes in Math. 819, Springer (1980)
158-174
\bibitem[Ga]{Ga} Gabai, David \emph{Taut foliations of $3$-manifolds and suspensions of $S^1$.} Ann. Inst. Fourier (Grenoble) 42 (1992), no. 1-2, 193-208.
\bibitem[GO]{GO} Gabai, David; Oertel \emph{Ulrich Essential laminations in 3-manifolds.} Ann. of Math. (2) 130 (1989), no. 1, 41-73.
\bibitem[Gh1]{Gh1} Ghys, E \emph{Flots d¡¯Anosov sur les 3-vari\'et\'es fibr\'ees en cercles.} Ergodic Theory Dynam.
Systems 4 (1984) 67-80
\bibitem[Gh2]{Gh2} Ghys, E \emph{Flots d¡¯Anosov dont les feuilletages stables sont diff\'erentiables.} Ann. Sci.
¨¦cole Norm. Sup. 20 (1987) 251-270
\bibitem[Ghr]{Ghr} Ghrist, Robert W \emph{Branched two-manifolds supporting all links.} Topology 36 (1997), no. 2, 423-448.
\bibitem[Go]{Go} Goodman, G \emph{Dehn surgery on Anosov flows.} from ¡°Geometric dynamics¡± (J Palis, Jr,
editor), Lecture Notes in Math. 1007, Springer (1983) 300-307
\bibitem[HP]{HP} Hammerlindl, A.; Potrie, R. \emph{Pointwise partial hyperbolicity in three-dimensional nilmanifolds.} J. Lond. Math. Soc. (2) 89 (2014), no. 3, 853-875.
\bibitem[HT]{HT} Handel, M;   Thurston, W, P. \emph{Anosov flows on new three manifolds.} Invent. Math. 59
(1980) 95-103
%\bibitem[Pla]{Pla} Plante, J. F. \emph{Anosov flows.} Amer. J. Math.  94  (1972), 729-754.
%\bibitem[Pl1]{Pl1} Plante, J. F. \emph{Foliations of $3$-manifolds with solvable fundamental group.} Invent. Math. 51 (1979), no. 3, 219-230.
\bibitem[Pl]{Pl}   Plante, J. F. \emph{Anosov flows, transversely affine foliations, and a conjecture of Verjovsky.} J. London Math. Soc. (2) 23 (1981), no. 2, 359-362.
 %\bibitem[Pot]{Pot}   Potrie, Rafael. \emph{Partial hyperbolicity and foliations in $\TT^3$.} J. Mod. Dyn. 9 (2015), 81-121.
 \bibitem[Ply]{Ply} Plykin, R. V.  \emph{On the geometry of hyperbolic attractors of smooth cascades.} Usp.Math. Nauk 39, No 6, (1984) 75-113 [English Transl. Russ. Math. Survey 39, No 6, p. 85¨C131].
  \bibitem[PW]{PW}    Palis, Jacob, Jr.; de Melo, Welington \emph{Geometric theory of dynamical systems.} An introduction. Translated from the Portuguese by A. K. Manning. Springer-Verlag, New York-Berlin, 1982. xii+198 pp.
 \bibitem[Ro]{Ro}   Rolfsen, Dale. \emph{Knots and links.} Mathematics Lecture Series, No. 7. Publish or Perish, Inc., Berkeley, Calif., 1976.
 \bibitem[Rua]{Rua}    Ruas, G. \emph{Atratores hiperb\'olicos de codimensao um e classes de isotopia em superficies.} Informes de Matem\'atica, I.M.P.A. S\'erie F-01/82, (1982)
 \bibitem[Sm]{Sm}   S Smale, \emph{Differentiable dynamical systems.} Bull. Amer. Math. Soc. 73 (1967) 747-817
 \bibitem[So]{So}  Solodov, V. V. \emph{Components of topological foliations.} (Russian) Mat. Sb. (N.S.) 119(161) (1982), no. 3, 340-354, 447.
 \bibitem[Thu]{Thu}   Thurston, William P. A norm for the homology of $3$-manifolds. Mem. Amer. Math. Soc. 59 (1986), no. 339, i-vi and 99-130.
\end{thebibliography}
\end{document}